\documentclass[11pt]{amsart}
\usepackage{amsmath,amsthm,verbatim,amssymb,amsfonts,amscd,blindtext,  graphicx,tensor,chngcntr}
\usepackage{xcolor}
\usepackage[utf8x]{inputenc}
\usepackage{mathtools}
\usepackage{arcs}
\usepackage{graphics}

\usepackage{euscript, enumerate}
\usepackage{url,hyperref}

\newcommand{\R}{{\mathbb R}}

\newcommand{\p}{\partial}
\newcommand{\fr}{\frac}
\newcommand{\la}{\langle}
\newcommand{\ra}{\rangle}

\newcommand{\e}{\epsilon}

\newcommand{\be}{\begin{equation}}
\newcommand{\ba}{\begin{aligned}}
\newcommand{\bee}{\begin{equation*}}
\newcommand{\ee}{\end{equation}}
\newcommand{\ea}{\end{aligned}}
\newcommand{\eee}{\end{equation*}}

\newcommand{\dist}{{\rm dist}\, }

\newcommand{\diam}{\operatorname{diam}}

\theoremstyle{plain}
\newtheorem{theorem}{Theorem}[section]

\newtheorem{lemma}[theorem]{Lemma}
\newtheorem{proposition}[theorem]{Proposition}
\newtheorem{prop}[theorem]{Proposition}

\theoremstyle{remark}

\newtheorem{remark}[theorem]{Remark}

\theoremstyle{definition}
\newtheorem{definition}[theorem]{Definition}

\numberwithin{equation}{section}

\title{Convergence of Curve Shortening Flow to Translating Soliton}
\author{Beomjun Choi}
\address{{\bf Beomjun Choi:} 
Department of Mathematics, POSTECH, 77 Cheongam-Ro, Nam-Gu, Pohang, Gyeongbuk 37673, Republic of Korea.}
\email{bchoi@postech.ac.kr}
\author{Kyeongsu Choi}
\address{{\bf Kyeongsu Choi:} Korea Institute for Advanced Study, 85 Hoegiro, Dongdaemun-gu, Seoul 02455, Republic of Korea.}
\email{choiks@kias.re.kr}
\author{Panagiota Daskalopoulos}
\address{{\bf Panagiota Daskalopoulos:} Department of Mathematics, Columbia University, 2990 Broadway, New York,  NY 10027, USA.}
\email{pdaskalo@math.columbia.edu}

\begin{document}
\title[Uniqueness of ancient GCF ovals]  {Uniqueness of ancient solutions \\ to Gauss curvature flow  asymptotic to \\a  cylinder}

\begin{abstract} We address the classification of ancient solutions to the Gauss curvature flow under the assumption that the solutions are contained 
in  a cylinder of bounded cross-section. For each cylinder of convex bounded cross-section, we show that there are only two ancient solutions which 
are asymptotic to this cylinder: the non-compact translating soliton and the compact oval solution obtained by gluing two translating solitons approaching each other  from time $-\infty$ from  two opposite ends. 

%
%

\end{abstract}
\maketitle
\tableofcontents

\section{Introduction}

A one-parameter family $\Sigma_t:=F(M^n,t)$ of complete convex embedded hypersurfaces defined by $F:M^n\times [0,T) \to \mathbb{R}^{n+1}$  is a solution of the {\em Gauss curvature flow} (GCF in abbreviation) if $F(p,t)$ satisfies
\begin{equation}
\tfrac{\partial}{\partial t}F(p,t)=-K(p,t) \, \nu (p,t),
\end{equation}
where $K(p,t)$ is the Gauss curvature of $\Sigma_t$ at $F(p,t)$, and $\nu(p,t)$ is the unit normal vector of $\Sigma_t$ at $F(p,t)$ pointing  outward of the convex hull of $\Sigma_t$. 

\smallskip

The GCF was first introduced by W. Firey \cite{Fi} in 1974 as a model that  describes  the deformation of a compact convex body $\hat \Sigma_0$ embedded in 
$R^{n+1}$  which is subject to wear under impact from any random angle.  An example can be a stone on a beach impacted by the sea. 
The probability of impact at any point $P$  on the surface $\Sigma_t = \partial \hat \Sigma_t$  is proportional to the Gauss curvature $K$ 
of $\Sigma$ at $P$.  W. Firey showed, assuming that a solution exists,  that the GCF   shrinks smooth, compact, strictly convex and centrally symmetric hypersurfaces embedded in $\R^{3}$ to round points.   The existence of solutions in any dimension was established in 1985  by Tso \cite{Ts}. 
Tso showed that   under the assumption that the  initial surface $ \Sigma_0$  is smooth, compact and strictly convex    the Gauss curvature flow admits a unique solution $\Sigma_t $ which shrinks to a point at the exact time $T^* : = V/4\pi$,    where $V$  is the volume enclosed by the initial surface $\Sigma_0$. Around the same time Chow \cite{Ch85} proved that, under certain restrictions on the second fundamental form of the initial surface, the Gauss curvature flow shrinks smooth compact strictly convex hypersurfaces to round points.
Later,  in \cite{An99} Andrews showed that the Gauss Curvature flow shrinks any compact convex hypersurface  in $\R^3$ to a round point.  For higher dimensions $n\ge3$, P. Guan and L. Ni in \cite{GN} obtained the convergence of the flow after rescaling  to  a self-shrinking soliton.
K. Choi and P. Daskalopoulos \cite{CD} have recently shown  that the sphere is the unique self-shrinking soliton which combined with the result in \cite{GN} shows that the only finite time singularities in the n-dimensional  GCF  are the spheres. 

\smallskip 

In this work we will study the {\em ancient solutions} to the GCF, that is solutions 
which exist for all time $t\in(-\infty,T)$, for some $T\in(\infty,\infty]$.  Shrinking and  translating solitons  are typical  important models  of ancient solutions. A shrinking soliton refers a solution which homothetically shrinks to a point. A  shrinking soliton which shrinks to  at spatial origin at time $t=0$, 
is of the form $\Sigma_t = (-t)^{\frac{1}{1+n}} \Sigma_{-1}$ and $\Sigma_{t}$ satisfies  $K= \frac{1}{(n+1)(-t)}\la F, \nu \ra $. A translating soliton refers 
to a solution which moves by translation $\Sigma_t = \Sigma_0 + \omega t $, along a fixed direction   $\omega\in \mathbb{R}^{n+1}$, 
it is defined for all  $t\in(-\infty,\infty)$, and satisfies $K =  \la -\nu, \omega \ra$. 


\smallskip 
In one dimension, the GCF  coincides with  the {\em Curve Shortening Flow} (CSF in abbreviation)  for curves embedded in $\R^2$. In  this case, there is only one translating soliton (up to isometries and  rescaling). It is called the {\em Grim Reaper solution}  and is given by the  graphical representation
$y=t -\ln \cos x $, for  $(x,t)\in(-\pi/2,\pi/2)\times (-\infty,\infty)$. 

\smallskip 

The result of P. Daskalopoulos, R. Hamilton and N. Sesum \cite{DHS} reveals that  shrinking and translating solitons are major building blocks of ancient solutions. It was shown in \cite{DHS} that  the only  compact convex ancient solutions to  CSF are the  shrinking round sphere or the {\em Angenent oval}.
The latter,   looks as if it is constructed by the gluing of two Grim Reapers coming from opposite ends.  It  is given 
 by the  implicit equation $\cos x =e^t \cosh y $ and, as $t\to -\infty$, it is approximately the  intersection of the {\em two Grim Reapers}  $y=(t -\ln2) -\ln \cos x$ and $y= -(t-\ln2) + \ln \cos x$.  Moreover recently, T. Bourni, M. Langford, and G. Tinaglia \cite{BLT} completed this classification by showing that the shrinking circle, the Angenent oval, the Grim Reaper, and the stationary line are the only convex ancient solutions to the curve shortening flow. 

\smallskip 
A  similar classification holds true  in the two-dimensional   {\em Ricci flow}. It was shown  by P. Daskalopoulos, R. Hamilton, and N. Sesum  in \cite{DHS2},
  that the  shrinking sphere and the King  solution,  are the only ancient solutions defined  on $\mathbb{S}^2$. 
For the {\em mean curvature flow}  and  the {\em Ricci flow}  in higher dimensions, the classifications are done under a  non-collapsing  along with convexity,  low entropy or certain  other conditions. See \cite{BC}, \cite{ADS}, \cite{B}, \cite{DS}. 

\smallskip 

In this paper, we prove the classification of ancient solutions to the Gauss curvature flow under the assumption that the solution is contained in a cylinder of bounded cross-section. The  relevance of this assumption is  found in the classification of  translating solitons to the GCF given by J. Urbas \cite{Ur1,Ur2} where each translator is shown to be a graph on a convex bounded domain $\Omega\subset \mathbb{R}^{n}$. Note also that the Grim Reaper and the Angenent oval are solutions to the curve shortening flow contained in a strip, and the rotationally symmetric steady cigar soliton and the King  solution are ancient solutions of the 2-dim Ricci flow asymptotic to a fixed round cylinder.  However, a significant difference between those previous results and ours in this work, is that a  translator exists for  each $\Omega\subset \mathbb{R}^n$ (see  J. Urbas \cite{Ur1,Ur2}) and hence there are {\em infinitely many ancient solutions}. We will show the {\em uniqueness}  of ancient solutions having {\em asymptotic cylinder}  $\Omega 
 \times \R$, according to the definition below.

\begin{definition}[Asymptotic cylinder of an ancient solution]  \label{def-asymptcy}
Assume that $\Sigma_t$, $t \in (-\infty, T)$, $T \in (-\infty, +\infty]$  is  a complete ancient GCF solution  such that  $\cup_{t \in (-\infty, T) }  \Sigma_t \subset \mathrm{cl}( \Omega )\times \R$, for a some open bounded domain $\Omega \subset  \R^n$,  and that  $\cup_{t \in (-\infty, T) }  \Sigma _t $ is   not contained in any smaller  cylinder  $\mathrm{cl}(\Omega') \times \R$.  We will then refer to the cylinder $\Omega \times \R$ as the {\em asymptotic cylinder} of the ancient solution $\Sigma_t$,  $t \in (-\infty, T)$. 
\end{definition}  Similarly, we define the asymptotic cylinder of a time slice $\Sigma_{t'}$ by the smallest cylinder containing $\Sigma_{t'}$.  
\smallskip

Our first result given below settles the {\em uniqueness}  of {\em non-compact}    ancient solutions.

\begin{theorem}[Uniqueness  of non-compact ancient solutions]\label{thm-noncompact} Given a convex bounded $\Omega\subset \mathbb{R}^n$, the translating soliton asymptotic to $\Omega\times\mathbb{R}$ is the unique non-compact ancient solution asymptotic to $\Omega\times \mathbb{R}$. This uniqueness 
holds up to  translations along  the  $e_{n+1}$ direction and reflection about $\{x_{n+1}=0\}$.
\end{theorem}

\smallskip 
Regarding {\em  compact ancient solutions}, our next result shows the  existence of {\em ancient oval solutions}  which  are the  analogue of the Angenent oval solution  for  {\em curve shortening flow. } 

\begin{theorem}[Existence of ancient oval solutions] \label{thm-ancientexistence}  Given a convex bounded $\Omega \subset \R^n$ with $C^{1,1}$ boundary, there exists a compact ancient solution $\Gamma_t \subset \R^{n+1}$ to  Gauss curvature flow which is defined  for all $t \in (-\infty ,T)$,  
it   becomes  extinct at   $T:= -\frac{2V_{\Omega}}{\omega_n}$,  
and has asymptotic cylinder $\Omega \times \R$. Here $V_\Omega$ is the volume under the graph of the translating soliton asymptotic to $\Omega\times \mathbb{R}$ which   is finite according to  Lemma \ref{lem-finitevol}. Furthermore, the solution $\Gamma_t$ satisfies the properties in Propositions   \ref{prop-existence} and  \ref{prop-23}. \end{theorem} 

Our last result shows that $\Gamma_t$ in Theorem \ref{thm-ancientexistence} is unique compact ancient solution asymptotic to $\Omega \times \R$. 
\begin{theorem}\label{thm-compactunique} Given a convex bounded $\Omega\subset \mathbb{R}^n$ with $C^{1,1}$ boundary,
let $\Sigma_t$, $t\in (- \infty, T)$ be a compact ancient solution to the Gauss curvature flow in $\mathbb{R}^{n+1}$ which is asymptotic to $\Omega \times \R$. 
Assuming that the solution becomes extinct  at time $T:=-\frac{2V_{\Omega}}{\omega_n}$,  there is $v\in \R$ such that $\Sigma_t + v e_{n+1} =\Gamma_t$,  for all $t \in  (- \infty, T)$, where $\Gamma_t$ is the solution constructed in Theorem \ref{thm-ancientexistence}. \end{theorem}

\begin{remark} In the proofs of  Theorem \ref{thm-ancientexistence} and Theorem \ref{thm-compactunique}, the only use of the  $C^{1,1}$ assumption on $\Omega$ is to ensure that $V_\Omega$ is finite. In other words, if one can show Lemma \ref{lem-finitevol} without such an assumption, this condition can be removed from the theorems. In two theorems, the extinction time $T=-\frac{2V_\Omega}{ \omega_n}$ is chosen so that the maximum height of $\Gamma_t$, $h(t) := \max_{x\in \Gamma_t}  |x_{n+1}| $, satisfies $h(t)=\lambda|t| + o(1)$ as $t\to-\infty$. Here $\lambda :=\frac{\omega_n}{2|\Omega|}$ is the {\em speed of the translating soliton}  asymptotic  to $\Omega \times \mathbb{R}$. 
\end{remark}

\smallskip

The organization of this paper is as follows:  
In Section \ref{sec-Preliminaries}, we give some preliminary results  and define an appropriate notion of  weak solution. 
This is  needed as  translating solitons defined on non-strictly convex domains are not necessarily smooth and  the corresponding  compact ancient solutions  may not be smooth as well. Theorem \ref{thm-noncompact}-\ref{thm-compactunique} will be shown in later sections with this notion of weak solution. 

In Section \ref{sec-asymp}, we show Theorem \ref{thm-asymptotic}, the asymptotic convergence of an ancient solution to 
a  translating soliton, as $t\to-\infty$, if one translates  the solution so that its  tip is fixed. In  their recent recent work \cite{CCD, CCD2} the authors 
established the  forward in time convergence of non-compact solutions asymptotic to a cylinder  to the corresponding  translating soliton. Theorem \ref{thm-asymptotic}
 is obtained as an application of this  result,   and hence we will refer to the results in  \cite{CCD2} when they are needed. 
Theorem \ref{thm-noncompact}, the uniqueness of non-compact ancient solutions, will be shown in Section \ref{sec-uniquenoncpt} as a consequence of Theorem \ref{thm-asymptotic} and \cite{CCD2}. 

In Section \ref{sec-existence}, we show Theorem \ref{thm-ancientexistence}, the existence of a compact ancient solution  asymptotic to a given cylinder, by showing Propositions   \ref{prop-existence} and \ref{prop-23}. These two propositions also establish  additional properties of the constructed solution $\Gamma_t$ which will be used when we show the uniqueness Theorem \ref{thm-compactunique} by comparing $\Gamma_t$ with an arbitrary ancient solution  $\Sigma_t$.

In Section \ref{sec-noncompact}, we show Theorem \ref{thm-compactunique}, the uniqueness of a compact ancient solution asymptotic to a given cylinder.  Part of   our  proof is inspired by the recent significant  works of  Bourni-Langford-Tinaglia \cite{BLT0,BLT} where they use the rate change in the enclosed the volume as a function of  time to estimate the location of the tips.

\section{Preliminaries}\label{sec-Preliminaries}
In this section we collect some preliminary results. Throughout this paper,  $h_{ij}$ denotes  the second fundamental form.  For a strictly convex solution, one may  consider the  inverse $b^{ij}$
of the second fundamental form $h_{ij}$,  which satisfies   $b^{ik}h_{kj} =\delta^i_j$. Let us recall the unique existence of translating solitons by J. Urbas and denote them as follows.
\begin{definition} [Theorem of J.Urbas \cite{Ur1, Ur2}] \label{def-u_Omega} Given a convex bounded domain $\Omega\subset \mathbb{R}^n$, we define $u_{\Omega}: \Omega \to \mathbb{R}$ by the graph function of the unique translating soliton which is asymptotic to $\Omega \times \mathbb{R}$, moves in the positive $e_{n+1}$ direction, and satisfies  $\inf u_\Omega(\cdot)=0$. In other words, the hypersurface given  by $\p \{ (x,x_{n+1})\in\mathbb{R}^{n+1}\,:\, x_{n+1} > u(x)\}$ defines  the translating soliton. The existence and the uniqueness is shown in \cite{Ur1, Ur2}. 
\end{definition}

\begin{remark}\label{remark-soliton} 
In the case where $\Omega$ is not a strictly convex domain, it  is possible that   $\limsup_{x\to x_0} u_\Omega(x)<\infty$,  for some $x_0\in \p\Omega$,
hence  the hypersurface    $\{x_{n+1}=u_\Omega(x)\}$ may  not be always complete. This is the reason why we denote the translating soliton by  $\p \{ x_{n+1} > u(x)\}$.  Urbas \cite{Ur2} showed the  existence of such solitons and their uniqueness among solutions realized in certain  generalized sense.  To be more specific,  Urbas \cite{Ur2} showed  if a convex function $u(x)$ defined  on $\Omega$ satisfies the translating soliton equation \begin{equation}
\det D^2u= \beta \, (1+|Du |^2)^{\fr{n+1}{2}}
\end{equation}
for some $\beta>0$ in the sense of Alexandrov,  and $|\mathbb{R}^n-  Du(\Omega)|=0$, then $u=u_\Omega +C$, for some  constant $C$. {\em We will use this characterization  of soliton in the proofs  of  Theorem \ref{thm-asymptotic}. }
\end{remark}

\begin{definition} \label{def-lambda}For  given convex bounded domain $\Omega \subset \mathbb{R}^n$, let us note the speed of associated translating soliton by \be\label{eq-lambda}\lambda :=  \fr{1}{|\Omega|}\left[\int_{\mathbb{R}^n} \fr{1}{(\sqrt{1+|p|^2})^{n+1}}dp \right]= \frac{\omega_n}{2|\Omega|}\ee where $\omega_n=|\mathbb{S}^n|$. One can find a derivation of this $\lambda$ in \cite{Ur2} or the equality case of  \eqref{eq-speedlambda'}.
\end{definition}

In \cite{CDL} it was shown that a translating soliton may  be weakly convex with  flat sides in which case it fails to be smooth at the boundary $\partial \Omega \times \R$
of its asymptotic cylinder. This requires a suitable notion of weak solutions. In Definition \ref{def-weaksol}, we define a weak solution in such a way that it satisfies global comparisons with smooth classical solutions. The  existence and the uniqueness of ancient solutions will then be shown in this  class of weak solutions. Here and the remaining sections, an ancient solution is assumed to be a {\em weak  solution to the GCF} in the sense of Definition \ref{def-weaksol} unless otherwise stated.  Also, throughout the paper,  we  will use $\hat \Sigma$ to denote the closed region which is bounded by $\Sigma$.

\begin{definition}[Definition 2.6 in \cite{CCD2}] \label{def-weaksol}
Suppose that $\hat \Sigma_t \subset \mathbb{R}^{n+1}$ for $t \in [T_0,T_1]$ is a one-parameter family of closed convex sets with non-empty interior.  $\Sigma_t =\p \hat \Sigma_t \subset \mathbb{R}^{n+1}$  is a {\em weak subsolution}  to the GCF if the following holds: for given a smooth strictly convex solution to the GCF $\Sigma'_t=\partial \hat \Sigma'_t$ defined for $t \in [a,b]\subset [T_0,T_1]$ with initial data $\hat \Sigma'_a \subset \hat \Sigma_a$,   $\hat \Sigma'_t \subset   \hat \Sigma_{t}$ holds for all $t\in[a,b]$. We define a {\em weak supersolution}in a similar way with the opposite inclusion.  $\Sigma_t=\p \hat \Sigma_t$ is a {\em weak solution}  if it is both a weak sub- and super-solution. $\Sigma_t= \hat \Sigma_t$ for $t\in(-\infty,T)$ is a {\em weak ancient solution} if $\Sigma_t$, $t\in[a,b]$, is a weak solution for all  $-\infty <a\le b <T$.

\end{definition}

The following result shows the existence and uniqueness of a weak solution starting at any convex hypersurface $\Sigma_0= \p \hat \Sigma_0 \subset \R^{n+1}$ which is compact or non-compact and asymptotic to a cylinder.

\begin{theorem} [Theorem 2.7 in \cite{CCD2}] \label{thm-exist1} Let $\hat \Sigma_0 \subset \R^{n+1}$ be a  convex set with non-empty interior. If  $\Sigma_0=\p\hat \Sigma_0$ is compact then there is a unique weak solution $\Sigma_t$ to the GCF running from  $\Sigma_0$ and defined over $t \in [0,T)$ for some $T <+\infty$. If $\Sigma_0=\p \hat \Sigma_0$ is non compact and asymptotic to a cylinder $\Omega\times \mathbb{R}$, then there is a unique weak solution $\Sigma_t$  to the GCF running from $\Sigma_0$ defined for all $t\in [0,+\infty)$. Moreover, each time slice  $\Sigma_t$ is non-compact and asymptotic to $\Omega\times \mathbb{R}$  for all $t\in[0,\infty)$. \end{theorem}

  \begin{proof} Let us only give  an outline of the proof and refer the reader to \cite{CCD2} for details.   Let $\hat \Sigma_{i,0}$ be a sequence of convex sets with smooth strictly convex boundaries which strictly increases to $\hat\Sigma_0$.  We denote by  $\Sigma_{i,t}$   the unique GCF  solution with initial data  $\Sigma_{i,0}=\p \hat \Sigma_{i,0}$ and we simply define $\Sigma_t$ to be the  limit of $\Sigma_{i,t}$.  Since each  $\Sigma_{i,t}$ is smooth it can be compared 
  with any  smooth solution or any  weak solution.  The existence and  the uniqueness of a weak solution, as stated in the theorem,  follow from this comparison principle. 
  \end{proof}
  
  \begin{lemma}\label{lem-26} The following hold:  \begin{enumerate}[{\em (i)}] 
  
\item The limit of any monotone sequence of weak solutions is a weak solution,  provided that the limit is compact.
    \item The comparison principle  between weak solutions holds. 
  \item  $\p_t \text{Vol}(\hat \Sigma_t) = -|\mathbb{S}^n|=- \omega_n$,   for any compact weak solution $\Sigma_t$. 
  \end{enumerate}
  
  \begin{proof} Let us begin by showing (i). 
Suppose $\p \hat \Sigma_{i,t}$ is an increasing sequence of weak solutions and $\hat \Sigma_t=\mathrm{cl}(\cup_i \hat \Sigma_{i,t})$. Then $
\Sigma_t:=\p \hat \Sigma_t$ is  a supersolution. 
To show that $\Sigma_t$ is a subsolution, suppose that  $\hat \Sigma'_a \subset  \hat \Sigma_a$ and assume, without loss of generality, that $0\in \mathrm{int}(\hat \Sigma'_a)$.  For $\lambda<1$, $\lambda \Sigma'_{a+ \lambda^{-(n+1)}({t-a})} $ is a smaller solution and thus $\lambda \hat \Sigma'_{a}\subset \hat \Sigma_{i,a} $ for large $i>i_\lambda$ (here we use the compactness of $\Sigma_t$). This shows that $\lambda \hat \Sigma'_{a+\lambda^{-(n+1)}({ t-a})}\subset \hat \Sigma_{t}$. By letting $\lambda \uparrow1$ we obtain $\hat \Sigma'_t \subset \hat \Sigma_t$.  
A similar  argument shows the same  for decreasing sequences. (ii) and (iii)  follow  by approximation with smooth solutions as  discussed in the proof of Theorem \ref{thm-exist1}. 
    \end{proof}
  \end{lemma}
Lemma \ref{lem-26}, in particular, implies that in the case that  $\Sigma_0$ is compact, 
the maximal time of existence in Theorem \ref{thm-exist1} is given by $T=\text{Vol}(\hat \Sigma_0)/\omega_n $. 
The following is the Harnack inequality for the GCF and its consequence to graphical solutions.
\begin{theorem}[B. Chow \cite{Ch91}, Proposition 3.2 \cite{CCD2}]\label{thm-Chow} Let $\Sigma_t$, $t\ge0$, be a smooth compact strictly convex solution to the {\em GCF}. Then,  
\be\label{eqn-harnK}
\fr{1}{ K}(\p_t K - b^{ij}\nabla_iK \nabla_j K  )\ge -\fr{n}{1+n} \fr{1}{t}.
\ee

Let $x_{n+1}=u(x',t)$, be a smooth strictly convex graphical solution to the {\em GCF} which could be possibly incomplete. If the solution satisfies \eqref{eqn-harnK},  
then \be\label{eqn-utt22}
 u_{tt} \ge - \fr{n}{1+n}\fr{u_t}{t}\ee and hence, for $t_2\ge t_1>0$, 
 \be\label{eqn-u123}
 u_t({\cdot,t_2})\ge \left( \fr{t_1}{t_2}\right)^{\fr{n}{1+n}}u_t(\cdot,t_1).\ee

\end{theorem}

We finish this section with the following regularity result, which roughly says that a weak GCF becomes smooth on region which are away
from the initial surface. A similar property holds for other degenerate equations, such as the porous medium equation. While this property is known to hold true for GCF as well,  its proof
 doesn't seem to exist  in the literature (see in \cite{DS} for a related result). We   include it here for completeness.  The proof uses some of main results shown in \cite{CCD2}.  We suggest to the reader to  skip the proof of this proposition at their first reading.

\begin{prop}\label{prop-harnack3} Assume that  $\Sigma_0=\p \hat \Sigma_0$ is  a  convex hypersurface which is either compact or non-compact asymptotic to a cylinder $\Omega\times \mathbb{R}$ of bounded cross-section. Let $\Sigma_t$ be a weak GCF solution  starting at   $\Sigma_0$.  If a point $p\in \Sigma_{t'}$,  for some $t'>0$,  is away from $\Sigma_0$ and $\hat \Sigma_{t'}$ has non empty interior, then the solution $\Sigma_t$ is  strictly convex and smooth around $p$ in spacetime,  i.e. there is $B_r(p)\subset \mathbb{R}^{n+1}$ such that $B_r(p)\cap \Sigma_T$ is smooth for $t\in(0,t']$. Moreover,   the Harnack inequality \eqref{eqn-harnK} holds on smooth part of $\Sigma_t$.  

\begin{proof} Consider a sequence smooth strictly convex compact solutions $\Sigma_{i,t}$ which approximates $\Sigma_t$ from the inside in the sense that $\mathrm{cl}({\cup_i\hat \Sigma_{i,t}}) =  \hat \Sigma_t$. See the proof of Theorem \ref{thm-exist1}, which is Theorem 2.7 in \cite{CCD2}, for the construction of such an approximation.  Let us denote the graphical representations of lower part of $\Sigma_{i,t}$ by $x_{n+1}=u_i(x,t)$. Since $\Sigma_{i,t}$ are compact solutions, $\Sigma_{i,t}$ and $u_{i}(x,t)$ satisfy Theorem \ref{thm-Chow}.

\smallskip 

Let $p$ be a point in $\Sigma_{t'} \setminus \Sigma_0$ for some $t'>0$.  After a translation  we may assume that $p = (0,0)\in\mathbb{R}^{n+1}$. Next, 
since a convex hypersurface is locally a convex graph, after a rotation and renaming the index $i$, we may find positive constants $r'$, $\delta$, $L$ and $M$ with the following significance: all $\Sigma_{i,t'}$ enclose the sphere of radius $r'$ centered at $(0,L)\in \mathbb{R}^{n+1}$ and the graphical representation $u_{i}(x,t)$ defined on   $D_{r'}(0)= \{x\in\mathbb{R}^n\,:\, |x|\le r'\}$ for $t\in[0,t']$. Moreover on $D_{r'}(0)$,  $-M \le u_i(x,0)\le -2\delta$ and  $-\delta \le u_i(x,t')$.  Our goal is to find some $\e >0$, $C_1$, $C_2$ such that 
\be \label{eq-goal11} 0<C_2 \le \p_t u_i(x,t) \le C_1, \qquad \text{ on } (x,t) \in D_{r'/2}(0) \times [t'-\e ,t'] . \ee
Once we have these graphical speed bounds. Proposition 4.3, Theorem 4.4, and  the argument in  Corollary 4.5 in \cite{CCD2} can be applied to $x_{n+1}=u_i(x,t)$ on $D_{r'/4}(0) \times [t'-\tfrac\e2 ,t']$ showing positive upper and lower bound on the curvature and gradient estimate (uniform in $i$).  This would give a uniform smooth estimate and would show the spacetime $C^{\infty}$  convergence of $\Sigma_{i,t}$ to $\Sigma_t$ around $p$.  
 
\smallskip 
The main tool in showing \eqref{eq-goal11} is the Harnack estimate. Let us fix a point $x'\in D_{r'/2}(0)$.    
 First, we are going to show the  upper bound for  $\p_t u_i(x',t')$. Using the  spherical  solution  starting from the sphere of radius $r'$ centered at $(0,L)\in \mathbb{R}^{n+1}$ at time $t=t'$  as a barrier, we may find some $t''>t'$  such that $u_i(x',t) \le L$ on $x'\in D_{r'/2}(0)$ and $t\le t''$. Inequality \eqref{eqn-u123} 
of  Theorem \ref{thm-Chow}, yields that for any $0<t<t''$, 
\[\ba M+L &\ge u_i(x',t'')-u_i (x',t) \\&= \int_{t}^{t''}\p_tu_i(x',s) ds  \ge \p_t u_i(x' , t)\int_{t}^{t''}   \left( \tfrac{t}{s}\right)^{\scriptstyle\fr{n}{1+n}} ds,\ea\] 
and it gives    \[\p_t u_i(x',t) \le \frac{M+L}{(1+n)\, t ^ {\frac{n}{1+n} }[(t'')^{\frac{1}{1+n}} -t ^{\frac{1}{1+n}}] }. \] In particular, for $t\in[t'/2,t']$, we have \[\p_t u_i(x',t) \le C_1, \quad \text{ for some }C_1=C_1(M+L,t',t'',n)\]
proving the upper bound in \eqref{eq-goal11}

\smallskip
Let us now show the  lower bound of $\p_t u_i(x',t)$ in \eqref{eq-goal11}. For $0\le \tau\le t'/2$, the previous upper bound implies 
\[\ba u_i(x',t'-\tau) &\ge u_i(x',t')- \tau C_1 \\&= u_i(x',0) + (u_i(x',t')-u_i(x',0)) -\tau C_1 \ge u(x',0) + \delta - \tau C_1. \ea\] 
Hence, for  $\tau \le \min(\tfrac{t'}2 ,  \tfrac{\delta}{2C_1})=:\e$, by integrating  inequality \eqref{eqn-u123} of  Theorem \ref{thm-Chow}, we obtain 
\[ \frac{\delta}{2}\le u(x',t'-\tau)-u (x',0)= \int_0 ^{t'-\tau}  \p_t u (x',s) ds \le (1+n) (t'-\tau) \, \p_t u(x',t'-\tau)\] 
which readily shows the  lower bound 
$$\p_tu_i(x',t) \ge \frac{\delta}{2(1+n)t'}=:C_2 , \quad \text{for} \,\,  t \in [t'-\e,t'].$$  This proves the desired estimate  which implies the smooth convergence of $\Sigma_{i,t}$
 to $\Sigma_t$  around the point $p$.
%
%
%
 \end{proof}
\end{prop}

\section{Convergence of solution around tip}\label{sec-asymp}

Throughout this section we will assume that $\Sigma_t$, $ t \in (-\infty, T)$ is a weak ancient complete  solution to the GCF  which is asymptotic to the cylinder $\Omega \times \R$,
as $t \to -\infty$ (see Definition \ref{def-asymptcy}). The goal is to show if we translate the solution and observe our solution around the tip region, then as $t\to-\infty$ it converges to the unique translating soliton asymptotic to $\Omega \times \mathbb{R}$.  

As we  mentioned earlier, an ancient solution may touch the boundary of its asymptotic cylinder $\Omega\times \mathbb{R}$ (c.f. in  \cite{CDL}).  For this reason,  some of the results in this section are written and shown in terms of $\mathrm{int}(\hat \Sigma_t)$. 



\begin{definition}\label{defn-hpm}
For an ancient convex GCF solution $\Sigma_t$, $t \in (-\infty, T)$, asymptotic to the cylinder $\Omega \times \R$, we define  $$h^+(t) := \sup_{x\in\Sigma_t} \langle x, e_{n+1} \rangle
\qquad \mbox{and} \qquad h^-(t) := \inf_{x\in\Sigma_t} \langle x, e_{n+1} \rangle$$to be the maximum and minimum heights, respectively. They are both finite if $\Sigma_t$ is compact. For the non-compact case, after  reflection, we will assume that $-\infty <h^-(t) <+\infty$   and $h^+(t)=\infty$. 

We also define  $p^+(t)$ and $p^-(t)\in\Sigma_t$ to be the tips of $\Sigma_t$ by the condition
\be\label{eqn-def-tips}
\langle p^+,e_{n+1}\rangle= h^+ \qquad \mbox{and} \qquad \langle p^-,e_{n+1}\rangle= h^-. 
\ee
In the non-compact case we only have one tip $p^-(t)$.

\end{definition}

\begin{figure}
\centering
\def\svgscale{1}{
\begingroup%
  \makeatletter%
  \providecommand\color[2][]{%
    \errmessage{(Inkscape) Color is used for the text in Inkscape, but the package 'color.sty' is not loaded}%
    \renewcommand\color[2][]{}%
  }%
  \providecommand\transparent[1]{%
    \errmessage{(Inkscape) Transparency is used (non-zero) for the text in Inkscape, but the package 'transparent.sty' is not loaded}%
    \renewcommand\transparent[1]{}%
  }%
  \providecommand\rotatebox[2]{#2}%
  \newcommand*\fsize{\dimexpr\f@size pt\relax}%
  \newcommand*\lineheight[1]{\fontsize{\fsize}{#1\fsize}\selectfont}%
  \ifx\svgwidth\undefined%
    \setlength{\unitlength}{282.75075287bp}%
    \ifx\svgscale\undefined%
      \relax%
    \else%
      \setlength{\unitlength}{\unitlength * \real{\svgscale}}%
    \fi%
  \else%
    \setlength{\unitlength}{\svgwidth}%
  \fi%
  \global\let\svgwidth\undefined%
  \global\let\svgscale\undefined%
  \makeatother%
  \begin{picture}(1,0.32001711)%
    \lineheight{1}%
    \setlength\tabcolsep{0pt}%
    \put(0.56105468,0.16254082){\color[rgb]{0,0,0}\makebox(0,0)[lt]{\lineheight{1.25}\smash{\begin{tabular}[t]{l}$h^+(t)$\end{tabular}}}}%
    \put(0,0){\includegraphics[width=\unitlength,page=1]{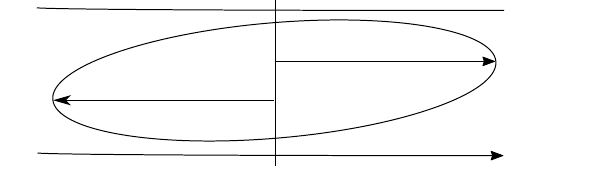}}%
    \put(0.22034503,0.16535812){\color[rgb]{0,0,0}\makebox(0,0)[lt]{\lineheight{1.25}\smash{\begin{tabular}[t]{l}$h^-(t)$\end{tabular}}}}%
    \put(0.4021515,0.01291217){\color[rgb]{0,0,0}\makebox(0,0)[lt]{\lineheight{1.25}\smash{\begin{tabular}[t]{l}$x_{n+1}=0$\end{tabular}}}}%
    \put(0.77906868,0.01884721){\color[rgb]{0,0,0}\makebox(0,0)[lt]{\lineheight{1.25}\smash{\begin{tabular}[t]{l}$x_{n+1}$\end{tabular}}}}%
    \put(0.83241943,0.16765388){\color[rgb]{0,0,0}\makebox(0,0)[lt]{\lineheight{1.25}\smash{\begin{tabular}[t]{l}$p^+(t)$\end{tabular}}}}%
    \put(-0.00266028,0.16122042){\color[rgb]{0,0,0}\makebox(0,0)[lt]{\lineheight{1.25}\smash{\begin{tabular}[t]{l}$p^-(t)$\end{tabular}}}}%
  \end{picture}%
\endgroup%

\caption{Definition \ref{defn-hpm}}}
\end{figure}

\begin{definition} \label{def-32}
Let $\Sigma_t = \partial \hat \Sigma_t$, $t \in (-\infty, T)$  be an ancient convex GCF solution,  asymptotic to the cylinder $\Omega \times \R$. For each $t \in (-\infty, T)$,   $\mathrm{int}(\hat \Sigma_t)$ can be represented as  the region between two graphs $u^+(\cdot,t)$ and $u^-(\cdot,t)$ defined on some  domain $\Omega_t \subset \Omega$ as follows: \bee  \mathrm{int}(\hat \Sigma_t) = \{ (x',x_{n+1}) \in \mathbb{R}^{n+1} \, :\, u^-(x',t) < x_{n+1} < u^+(x',t)\quad\text{and}\quad x'\in \Omega_t\}.\eee
Here, $\Omega_t$ is  the image of the projection of $\mathrm{int}(\hat\Sigma_t)$ to the hyperplane $\{x_{n+1} =0\}$, which is an open bounded convex set. For non-compact case, we set $u^+ = \infty$.  
\end{definition}

Note that since  $\Omega \times \R$ is the smallest open cylinder containing $\cup_t\mathrm{int}(\hat \Sigma_t)$, the domains $\Omega_t$ increase to $\Omega$,  as $t\to-\infty$. For the non-compact case, the last assertion in Theorem \ref{thm-exist1} implies $\Omega_t=\Omega$ for all $t$.  The functions  $u^+(\cdot, t)$ and $u^-(\cdot, t)$ are graphical solutions to the  GCF, which are defined  on the   domain $\Omega_t$. 
Before proceeding to the next theorem, recall the definition of translating soliton $u_\Omega$ in Definition \ref{def-u_Omega}  and the speed $\lambda$ in  Definition \ref{def-lambda}.

\begin{theorem}\label{thm-asymptotic} Let $\Sigma_t $, $t \in (-\infty, T),$ be a complete ancient convex weak  solution of GCF which is  asymptotic to the cylinder $\Omega \times \R$. 
Then, as  $t\to-\infty$, $\Sigma_t- h^{-}(t)e_{n+1}$ converges locally smoothly to the unique translating soliton $\{x_{n+1}=u_\Omega(x)\}$.  In the case that $\Sigma_t$ is compact, 
  $\Sigma_t- h^{+}(t)e_{n+1}$ also converges locally smoothly to the translating soliton $\{x_{n+1}=-u_\Omega(x)\}$.  More precisely, we have
 \[u^-(x,t)-h^-(t) \rightarrow u_\Omega(x) \,\, \, \text{and}\,\,\,  u^+(x,t)-h^+(t) \rightarrow -u_\Omega(x) \,\,\,  \text{in } C^\infty_{loc}(\Omega)\]
as $t\to-\infty$. 
\end{theorem}

\smallskip 

We need several lemmas before giving the  proof of this theorem. 

\begin{lemma} \label{lem-convex} We have $h^-(t) \to -\infty$ as $t\to -\infty$. If $\Sigma_t$ is compact, then $h^+(t) \to \infty$ as $t\to -\infty$ as well. Furthermore, in both cases we have  $\cup_{t}\mathrm{int}( \hat\Sigma_t)=\Omega\times \mathbb{R}$. 

\begin{proof} We give the proof assuming that  $\Sigma_t$ is compact. The proof in the non-compact case is similar. We first  show that $h^-(t) \to -\infty$ as $t\to -\infty$ by a  contradiction argument.  Suppose that $h_- \ge -C$,  for some $C<\infty$ for all time. This implies $\cup_t \mathrm{int}(\hat\Sigma_t )\subset \Omega\times [-C,\infty)$. Then we may find a translating soliton which is asymptotic to a slightly but strictly larger cylinder while containing $\Omega\times [-C,\infty)$. By comparing  this soliton with our solution $\Sigma_t$, starting at large negative times $t_0 \ll -1$, we conclude that  $\Sigma_t$ has to be empty for each $t$. 
This is a contradiction and hence $\lim_{t\to -\infty} h^-=-\infty$. Similarly, $\lim_{t\to -\infty} h^+=\infty$. 
\smallskip 

Let us now  see that  $\cup_{t} \mathrm{int}(\hat\Sigma_t)=\Omega\times \mathbb{R}$. 
 Since $\cup_{t}\mathrm{int}( \hat\Sigma_t)$ is a convex set and its boundary contains $p^+(t)$ and $p^-(t)$ which move to opposite infinities as $t\to-\infty$,  it easy to see that the sections $\cup_{t} \mathrm{int}(\hat\Sigma_t) \cap \{e_{n+1}=l\}$ have to be identical and thus $\cup_{t} \mathrm{int}(\hat\Sigma_t)$ is a convex open cylinder. Since, by assumption,  $\cup_{t} \mathrm{int}(\hat\Sigma_t)$ is contained in no smaller cylinder than $\Omega\times \mathbb{R}$,  we obtain the conclusion.
%
\end{proof}
\end{lemma}

\smallskip 
From now on, we will concentrate on the convergence of $u^-(x,t)-h^-(t)$. The convergence of $u^+(x,t)-h^+(t)$ follows  by a similar arguments. 
Lemma \ref{lem-convex} and Proposition \ref{prop-harnack3} imply the  following regularity lemma. 

\begin{lemma} [c.f. Theorem \ref{thm-Chow} or \cite{Ch91}]\label{lem-harnack} An  ancient weak solution $\Sigma_t$, $t \in (-\infty, T)$,  
satisfying the assumptions of Theorem \ref{thm-asymptotic}  is  smooth and strictly convex away from $\p \Omega \times \mathbb{R}$ provided that $\hat \Sigma_t$ has non empty interior. Moreover,  the Harnack inequality \be\fr{1}{ K}(\p_t K - b^{ij}\nabla_iK\nabla_j K  )\ge 0 \ee 
holds on all points where $\Sigma_t$ is smooth. 
As a consequence, we have
\begin{enumerate}[{\em (i)}]  \item $\p_t u^- (x,t)= \frac{K}{\langle- \nu , e_{n+1}\rangle}$ satisfies $\p^2_{tt}u^-(x,t)\ge 0 $ on $(x,t)\in\cup _t (\Omega_t\times\{t\})$. \item Let $K(\nu,t)$ be the Gaussian curvature of a point on $\Sigma_t$ parametrized by its outer unit normal $\nu$. Then $\p_t K (\nu_0,t)\ge 0$ whenever $\Sigma_t$ is smooth around the point $p_0$ with $\nu(p_0,t)=\nu_0$. \end{enumerate}
\end{lemma}

In the following steps, we are interested in establishing  lower bounds on the Gaussian curvature $K$ for any weak ancient solution $\Sigma_t$
satisfying the assumptions of the Theorem \ref{thm-asymptotic}.

\begin{lemma}\label{lem-beta} 
Let \be\label{eq-1}\beta:=\lim_{t\to-\infty} \p_th^-=\lim_{t\to-\infty} K(-e_{n+1},t)\ee 
which exists by  Lemma \ref{lem-harnack}.  Then,  we have $\beta \ge \lambda$, where $\lambda$ is the speed of the  translating soliton  in Definition \ref{def-lambda}.\end{lemma}

\begin{proof} 
We argue by contradiction. By translating the solution in time if necessary, we may assume $\hat \Sigma_0$ is not empty.  Suppose that $\beta<\lambda -2\e$,  for some $\e>0$. Consider   a strictly larger cylinder containing $\Omega\times \mathbb{R}$ whose corresponding translating soliton  has the speed $\lambda-\e$. Let us denote by $\bar\Sigma^{-}$ to be such a soliton moving in positive $e_{n+1}$ direction and having $\inf_{x\in\bar \Sigma^-} \langle x, e_{n+1} \rangle =0$ (namely its tip is the point $(0,0) \in \R^{n+1}$).
 Then there is $C_1>0$ such that $\bar \Sigma^- - C_1 \, e_{n+1}$ encloses  $\Omega \times [0,\infty)$. Therefore, $\bar\Sigma^-  + (-C_1+h^-(t))\,  e_{n+1}$ encloses  $\Sigma_t$ for $t<0$.  Then, the  comparison principle implies that  the surface  $\bar\Sigma^- +(-C_1+h^-(t)+(\lambda-\e)\, \tau) \,  e_{n+1}$ encloses  $\Sigma_{t+\tau}$ for $\tau \ge0$. By letting $\tau=-t$, we conclude that $\bar \Sigma^- +(-C_1+h^-(t) -(\lambda-\e) t) \, e_{n+1}$ encloses  $\Sigma_0$ for all $t<0$. On the other hand,  as $t \to -\infty$, $h^-(t) \ge  (\lambda-2\e)\, t+ o(t)$ by  \eqref{eq-1}. Thus $$-C_1+h^-(t) -(\lambda-\e) \, t \ge  -C_1 -\e \, t +o(t) \to \infty,
\quad  \text{ as } t\to -\infty,$$ which contradicts the assumption $\hat \Sigma_0$ is non-empty. \end{proof}

\begin{proposition} \label{prop-graphspeeds}Let $\Sigma_t$, $t \in (-\infty, T)$ be an  ancient solution 
satisfying the assumptions of the Theorem \ref{thm-asymptotic}. Given any  $\Omega' \subset \subset \Omega$, there is $t_0<0$ and $c>0$ such that 
\[c\le  \p_tu^-\le c^{-1},  \qquad \text{for }t\le t_0.\] 
\begin{proof} Let $\e >0$ be such that $\dist(\Omega',\p\Omega)=2\e>0$.
By Lemmas \ref{lem-convex} and \ref{lem-beta},  we may choose $t_0\ll  -1$ so that the following hold for $t\leq t_0$: 
\begin{enumerate}[(i)]
\item $\beta \, (t_0-t)  \le h^-(t_0)-h^- (t)\le 2\beta \, (t_0-t)$. 
\item If  $\Omega_{t,0}$ is the  cross-section of $\hat \Sigma_t$ at   $x_{n+1}=0$, namely we have 
($\hat \Sigma_t \cap \{\la x,e_{n+1} \ra =0\}=: \Omega_{t,0}\times\{0\}$),  then $\Omega'\subset \Omega_{t,0}$ and $\dist(\Omega',\p\Omega_{t,0})\ge \e$.
\end{enumerate} 

\smallskip 
From now on, assume $t\leq t_0$ and let $x'\in \Omega'$ be an arbitrary point. 
By the   monotonicity of $\p_t u^-(\cdot, t)$ in $t$ which follows from  \ref{eqn-utt22}, we have  $\p_t u^-(x',t)\le \p_tu^-(x',t_0) \leq  \sup_{\Omega'} \p_t u(\cdot, t_0) <\infty$, 
which  proves the upper bound. 

\smallskip

We next show the lower bound. Since $\Sigma_t$ is convex, it has to {enclose} a cone generated by the base $\Omega_{t,0}\times\{0\}$ and the vertex $p^-(t)$. Together with 
property (ii) above, this implies the bound  $u^-(x',t)\le \frac{\e}{C} h^{-}(t)$,   where $C=\diam \Omega$ (recall that both $u^-$ and $h^-$ are negative).
Using also that  $h^{-}(t) \le u^-(x',t)$,  
we conclude that  for any $\tau_1<\tau_2 \leq t_0$, we have  \[u(x',\tau_2) \ge h^{-}(\tau_2) \ge  2\beta\,  (\tau_2-t_0) +h^- (t_0)\]   
and 
\[u(x',\tau_1) \le \frac{\e}{C} h^{-}(\tau_1) \le    \frac{\e}{C}(\beta \, (\tau_1-t_0) +h^- (t_0)).\]
Subtracting these two inequalities and using $t_0 <0$ and $h^-(t_0) <0$ yields 
\bee
\begin{split}
u(x',\tau_2)-u(x',\tau_1) &\ge \beta \big( 2\tau_2 -  \frac{\e}{C}\tau_1 \big)-\beta(2-\frac{\e}{C})t_0+\big(1- \frac{\e}{C}\big)h^-(t_0) \\
&\ge\beta( 2\tau_2 -  \frac{\e}{C}\tau_1 \big)+h^-(t_0).
\end{split}
\eee If we choose $\tau_1= \tfrac{C  (2+L)}{\e} \tau_2$ for $L>0$, the monotonicity of $u^-_t$ implies \[u^-_t(x',\tau_2)\big(\tau_2-\tfrac{C(2+L)}{\e}\tau_2 \big)\ge u(x',\tau_2)-
u(x',\tfrac{C(2+L)}{\e} \tau_2) \ge   L  \beta(-\tau_2)+h^- (t_0)\]
which gives 
 \[u^-_t(x',\tau_2) \ge \tfrac{L\beta}{\frac{C\, (2+L)}{\e}-1}+   \tfrac{h^- (t_0)}{\tau_2- \frac{C(2+L)}{\e}\tau_2 }.  \]
Finally,  taking $L\to\infty$, we obtain the desired lower bound  ${u^-_t(x',\tau_2) \ge \tfrac{\e\beta}{C}.}$  

\end{proof}
\end{proposition}

\begin{prop}\label{prop-asymptotic} Let $\Sigma_t$, $t \in (-\infty, T)$, be an  ancient solution 
satisfying the assumptions of the Theorem \ref{thm-asymptotic}. For any  given sequence $\tau_i \to -\infty$,  passing to a subsequence if necessary,  $u^-_i(x,t):= u^-(x,t+\tau_i)-h^-(\tau_i)$ converges to $u_{\infty}(x,t)$ in $C^{\infty}_{loc}(\Omega\times \mathbb{R})$.  Moreover, the limiting  graphical solution $x_{n+1}=u_\infty(x,t)$ satisfies $\p_t u_\infty\equiv \beta $,  where $\beta$ is as in  \eqref{eq-1}.  The solution  $u_\infty$ represents a translating soliton which may possibly be incomplete. 
\begin{proof}
Since we have the bounds on the graphical speed $u_t= K\la -\nu,e_{n+1}\ra ^{-1}$ from Proposition \ref{prop-graphspeeds}, we can apply Proposition 4.3 and Theorem 4.4 in \cite{CCD2} as {they are} applied in Corollary 4.5 in \cite{CCD2} and obtain the following:  for any given $\Omega' \subset\subset \Omega$, there is $t_0<0$ and $C>0$ such that 
\[|Du^-|, \, \lambda_{\min}^{-1},\, \lambda_{\max} \le C\quad \text{on } \Omega'\times(-\infty,t_0].\] 

The equation of graphical GCF becomes uniformly parabolic provided we have the gradient bound, positive upper and lower curvature bounds. Thus we may pass to a limit  $u^-_i$, $i\to \infty$, by the standard regularity theory of parabolic equations, obtaining a graphical eternal solution $u_\infty(x,t)$. In view of (i) in Lemma \ref{lem-harnack}, $\p_t u_\infty(x,t)$ must be independent of $t$, that is  $\p_t u_\infty(x,t)=\beta_\infty(x)$. Furthermore, the fact that  $|Du_\infty (\cdot,t)|$ is bounded, 
globally in time,  on every compact set $\Omega' \subset \subset \Omega$, implies that $\beta_\infty(\cdot)$ has to be a constant.  From Lemma \ref{lem-beta}, we conclude that  $\p_t u_\infty \equiv \beta$.

\end{proof}   
\end{prop}

\begin{proof}[Proof of Theorem \ref{thm-asymptotic}]  By Proposition \ref{prop-asymptotic}, $u_\infty(x,t)=u_{\infty,0}(x)+\beta t $ where $u_{\infty,0}(\cdot):= u(\cdot, 0)$.  It  remains  to prove that $u_{\infty,0}(\cdot)=u_\Omega(\cdot)$. By the characterization of $u_\Omega$ given after Definition \ref{def-u_Omega},  it suffices to show $|\mathbb{R}^n -Du_{\infty,0}(\Omega)|=0$. 
Note that 
 \[ u_t \equiv \beta= (1+|Du_{\infty,0} |^2)^{\scriptscriptstyle\frac12}\left[ \fr{\det D^2u_{\infty,0}}{(1+|Du_{\infty,0}|^2)^{\scriptscriptstyle\fr{n+2}{2}}}\right]\, \text{ on }\,\Omega.\]
This implies
\be\label{eq-speedlambda'}
\ba \beta |\Omega|&= \int_{\Omega} \fr{\det D^2u_{\infty,0}}{\left({1+|Du_{\infty,0}|^2}\right)^{\scriptstyle\frac{n+2}{2}- \frac{1}{2}}}
= \int_{Du_{\infty,0}(\Omega)} \fr{1}{(\sqrt{1+|p|^2})^{\scriptstyle n+1}}\\ &\le\int_{\mathbb{R}^n} \fr{1}{(\sqrt{1+|p|^2})^{\scriptstyle n+1}} =\lambda|\Omega|.\ea\ee
We have shown $\beta\ge\lambda$ in Lemma \ref{lem-beta} and therefore the equality must hold in \eqref{eq-speedlambda'}. In particular, this shows $\beta =\lambda$ and  $|\mathbb{R}^n- Du_{\infty,0}(\Omega)| =0$.

\end{proof}
\section{Uniqueness of non-compact ancient solution} \label{sec-uniquenoncpt}

We are ready to give a proof of the uniqueness of non-compact ancient solution.  

\begin{proof}[Proof of Theorem \ref{thm-noncompact}] Let $\Sigma_t$ be an ancient solution as in the statement of the theorem.  By Theorem \ref{thm-exist1}, the solution exist for all $t\in(-\infty,\infty)$ and $\Sigma_t$ is asymptotic to $\Omega\times\mathbb{R}$ for each time slice. i.e. the domain of graphical representation of $\Sigma_t$ does not change over time. Therefore, we may represent $\Sigma_t$ by a graph $\p\{(x,x_{n+1})\in\mathbb{R}^{n+1}\,:\,x_{n+1}\ge u^-(x,t)\}$ for $(x,t)\in \Omega\times(-\infty,\infty)$. 

The main result, Theorem 1.1, in \cite{CCD2} proves the forward-in-time convergence to the soliton $u_\Omega(x)$, namely \[u^-(x,t)-h^-(t) \rightarrow u_\Omega(x) \quad \text{in } \,\,  C^\infty_{loc}(\Omega), \quad \text{ as } \, t\to\infty.\]
Let us fix an arbitrary $x'\in\Omega$. Together with Theorem \ref{thm-asymptotic},  \[\lim_{t\to\infty}u^-_t(x',t)=\lim_{t\to-\infty}u^-_t(x',t)=\lambda\] where $\lambda$ is the speed of 
the translating soliton $u_\Omega$.  By (i) Lemma \ref{lem-harnack}, $u^-_t\equiv  \lambda$ on $(x,t)\in\Omega\times \mathbb{R}$ showing  $\Sigma_t$ is a  translating soliton with the speed $\lambda$. We may repeat the same argument in \eqref{eq-speedlambda'}, while $\beta$ replaced by $\lambda$, to conclude $|\mathbb{R}^n-Du^-(\cdot,t)(\Omega)| =0$ and hence $u^-(x,t) = u_\Omega(x)+\lambda t + C$,  for some constant $C$.

\end{proof}


\section{Existence of compact ancient solution}\label{sec-existence}

Let $\Omega \subset \R^n$ is a bounded convex open domain. In this section we will construct an ancient compact solution of GCF which 
has asymptotic cylinder  $\Omega \times \R$, as $t \to -\infty$;  (see Definition \ref{def-asymptcy}).

We recall that $u_\Omega$ denotes the translator associated with the domain $\Omega$ satisfying $\inf_{x\in\Omega} u_\Omega(x)=0$.  For the construction of compact ancient solutions we need to show that  the volume under the translating soliton $V_\Omega:= \int_\Omega u_\Omega(x) \, dx$ is finite. Although  this is  expected to hold  for any compact convex  domain $\Omega$ with no further regularity assumptions on $\partial \Omega$, we could show this under $C^{1,1}$ boundary condition.


\begin{lemma} \label{lem-finitevol} Assume that  $\Omega \subset \R^n$ be a convex bounded open domain with  $C^{1,1}$ boundary. Let $x_{n+1}=u_\Omega(x)$, $x\in \Omega$, be the  translating soliton associated with the domain  $\Omega$ 
and having  $\inf _{x\in \Omega} u_\Omega(x) =0$.  Then, the volume under the translating soliton is finite, i.e. we have 
 \be\label{eqn-volume1}
 V_{\Omega}:=\int_{\Omega}u_\Omega(x)dx < \infty.\ee 
\smallskip 

\begin{remark} If this lemma is shown without $C^{1,1}$ assumption, it also proves Theorem \ref{thm-ancientexistence} and Theorem \ref{thm-compactunique} without $C^{1,1}$ assumption on $\p\Omega$. 
\end{remark} 
%
%
%
%

\begin{proof}[Proof of Lemma \ref{lem-finitevol}] The basic strategy is to find a supersolution of the graphical translating soliton equation \[ \dfrac{\det D^2\phi}{(1+|D\phi|^2)^{\scriptscriptstyle\frac{n+1}{2}}} \le \lambda = \dfrac{\omega_n}{2|\Omega|}\] which is integrable near the boundary of $\Omega$.

Assume first that $\p \Omega$ is smooth. Before going into the  details, let us recall some properties of the distance function   $d(x)$ from any point $x\in \Omega$ to $\p \Omega$. The function $d(x)$  is  well defined   on $\{ y\in \Omega\,:\, d(y,\p\Omega)<\lambda_0\}$,   $\lambda_0 := \sup_{y\in\p\Omega} \lambda_{\max}(y)$, where at each $y \in \Omega$,  $\lambda_{\max} (y):= \max_{i=1,\cdots n-1}  \lambda_i(y)$   denotes the maximum of the principal curvatures  of $\partial \Omega$ at $y$.  
Furthermore, $d(x)$ is a  smooth   function in this tubular neighborhood. 

\smallskip 
Let  $x\in\Omega$ be a point in this neighborhood and  $\pi(x)\in\p\Omega$ be the   point such that  $|\pi(x)-x|=\dist(x,\p\Omega)$.  If we denote  by $\lambda_i$, $i=1, \ldots, n-1$  the principal curvatures of  the hypersurface $\p\Omega\subset \mathbb{R}^n$ at $\pi(x)$, then with respect to the 
orthonormal basis $\{(e_i)_{i=1}^{n-1},-\frac{\pi(x)-x}{|\pi(x)-x|}\}$ of $\R^n$,  we have \[ D d(x) = (0,\ldots,0,1)\] and\[ D ^2d(x) =\begin{bmatrix}\begin{array}{c|c}
\begin{bmatrix} \text{diag}\left(-\frac{\lambda_i}{1-\lambda_i d}\right)\end{bmatrix}_{n-1\times n-1} & \begin{bmatrix} 0 \end{bmatrix}_{n-1\times 1} \\
\hline
\begin{bmatrix} 0\end{bmatrix}_{1\times n-1} & 0
\end{array}   
   \end{bmatrix}.\]

\noindent  Define   $\phi(x) = -L \, \log d(x) $ as our   test function. Then   in this neighborhood we have \be\begin{split}
 \label{eq-subsol}  \fr{\det D^2\phi}{(1+|D\phi|^2)^{\scriptscriptstyle\frac{n+1}{2}}} &= \frac{1}{(1+\frac{L^2}{d^2})^{\scriptscriptstyle\frac{n+1}{2}}} \frac{L^n}{d^{n+1}} \, \Pi_{n=1}^{n-1}\frac{ \lambda_i(\pi (x)) }{1- \lambda(\pi(x))\, d(x)} \\&\le  \frac1L \, \Pi_{n=1}^{n-1}\frac{ \lambda_i(\pi (x)) }{1- \lambda_i(\pi(x))\, d(x)}. 
 \end{split}\ee

\smallskip 

Assume next   that $\p \Omega$ is in $C^{1,1}$ and take  a strictly monotone increasing sequence  $\{\Omega_m\}$  of   convex domains  which approximates $\Omega$ from the inside in such a way that each $\p \Omega_m$ is  smooth and $\sup_{y\in \p \Omega_m} \lambda_{\max} (y) <2\lambda_0$. Here  strictly monotone  means $\Omega_m \subset\subset \Omega_{m+1}$. We may also assume that $\dist (x, \p\Omega) < \frac{1}{4\lambda_0}$ for all $x\in \p \Omega_1$. 
Set   $M:= \sup_{x\in \Omega_1} u(x)>0$ and define the functions 
$$\phi_m(x)= -\frac{(4\lambda_0)^{n-1}}{\lambda}\log \frac{(\dist(x, \p \Omega_m))}{\diam \Omega} +M, \qquad \mbox{for } \,\, x\in \Omega_m\setminus \Omega_1.$$ Then,  by our choice of $\Omega_m$, each $\phi_m$ is smooth in  the interior of $\Omega_m \setminus \Omega_1$. 
Furthermore, from \eqref{eq-subsol}   we have 
\[ \fr{\det D^2\phi_m}{(1+|D\phi_m|^2)^{\scriptscriptstyle\frac{n+1}{2}}} \le \frac{\lambda}{(4\lambda_0)^{n-1}} \left ( \frac{2\lambda_0}{1- \frac{2\lambda_0}{4\lambda_0}}\right)^{n-1} \le \lambda .\] 
 
\smallskip  
  
We will next  compare $\phi_n$ with $u_\Omega$ to conclude that $V_\Omega < \infty$. Since $\phi_n \ge u $ on $\p \Omega_1$ and  it becomes  infinite on $\p \Omega_m$, the comparison principle implies $\phi_m \ge u $ in  the interior of $\Omega_m \setminus \Omega_1$. Note that $\phi_m$ converges locally uniformly to $\phi:=  -\frac{(4\lambda_0)^{n-1}}{\lambda}\log (\dist(x, \p \Omega)) +M $ on $\Omega\setminus \Omega_1$, which implies $\phi(x) \ge u(x)$ in this region.  Since $\phi$ is integrable on $\Omega\setminus \Omega_1$, this implies $\int_\Omega u=V_\Omega$ is finite.

\end{proof}
\end{lemma}
 
 \smallskip
 
\begin{proof}[Proof of Theorem \ref{thm-ancientexistence}]

 Recall that the speed of the translator $u_\Omega$ defined on the domain $\Omega$ is given by $\lambda = \frac{\omega_n}{2|\Omega|}$. Theorem \ref{thm-ancientexistence} is implied by two propositions below.  

\smallskip

 \begin{prop} \label{prop-existence} Let  $\Omega \subset \R^n$ be a convex bounded domain with $C^{1,1}$ boundary.  Then, there is a compact weak ancient solution $\Gamma_t$ of the Gauss curvature flow, defined on $t\in(-\infty,T)$ with $T:= -\frac{2V_{\Omega}}{\omega_n}$, such that
  \begin{enumerate}[{\em (i)}]
 \item $\mbox{\em Vol} \, (\hat \Gamma_t) = - \omega_n t - {2V_{\Omega}}$, where $V_\Omega$ is given by \eqref{eqn-volume1};
 \item $\Gamma_t$ has reflection symmetry  with respect to $x_{n+1}=0$;
 \item $\Gamma_t$ is contained in $\Omega\times \mathbb{R}$, but not in a smaller cylinder,  i.e. $\Gamma_t$ is asymptotic to $\Omega\times \mathbb{R}$;
 \item $\Gamma_t$ is smooth in the interior of $\Omega \times \mathbb{R}$,   for $t<T$ and satisfies the differential Harnack inequality.  
 \end{enumerate}
 
 \begin{proof} For our  given bounded convex domain $\Omega$, denote by  $u_\Omega$ the graph of translating soliton corresponding to the domain $\Omega$
 having speed   $\lambda := \frac{\omega_n}{2|\Omega|}$ and satisfying $\inf_{\Omega} u_\Omega(x)=0$ (see  Definition \ref{def-u_Omega}). 
To simplify the notation, from now on we will denote $u_\Omega(x)$ simply by $u(x)$. 

\smallskip
The graphs $x_{n+1} = u(x)+\lambda \, t$ and $x_{n+1}=-u(x)- \lambda \, t$, $t \in \R$,  define translating solitons, moving in opposite directions and having tips at a distance $2\lambda \, |t|$ from each other.  The basic idea here is to construct our  solution $\Gamma_t$  as  limit  of hypersufaces which  for $t \ll -1$ 
are approximated by the boundary of the  region  $\{ x_{n+1} > u(x)+\lambda \, t \} \cap \{ x_{n+1}<-u(x)- \lambda \, t\}$. 

\smallskip 
To  make this rigorous, for any $s<0$  define $\hat \Gamma_{s,0}$ to be  the convex region which is bounded between the hypersurfaces $x_{n+1} = u(x)- \lambda \, |s|$  and $x_{n+1}= - u(x)+ \lambda \, |s|$. By Lemma \ref{lem-finitevol}, we can deduce that 
\be \label{eq:volume0}\big ( \text{Vol}(\hat \Gamma_{s,0})-\omega_n|s| \big ) \uparrow  - 2 V_\Omega, \qquad \mbox{as} \,\, s\to -\infty.
\ee
Note that $\p \hat \Gamma_{\tau,0}$,  viewed  as a  one-parameter family of hypersurfaces in $\tau$, is  a weak  subsolution to the GCF.
 
\smallskip 
Let $\Gamma_{s,t}$, $t\in[0,\tau_s)$ with $\tau_s = \text{Vol}(\hat \Gamma_{s,0})/\omega_n$, be the weak solution to the GCF starting  from $\Gamma_{s,0}$. Consider the time translated solutions $\Gamma_{s,t-s}$. $t\in[s,\tau_s+s)$. For each fixed $t$, we claim that $\Gamma_{s,t-s}$, $s\le t$, are monotone decreasing as $s\to -\infty$: for $s_1<s_2<0$, by the comparison principle between $\Gamma_{s_1,t}$ and $ \Gamma_{s_1+t,0}$, $\Gamma_{s_1,s_2-s_1}$ is contained in $\Gamma_{s_2,0}$. Again by the comparison principle between $\Gamma_{s_1,s_2-s_1+\tau}$ and $\Gamma_{s_2,\tau}$, we conclude $\Gamma_{s_1,t-s_1}$ is contained in $\Gamma_{s_2,t-s_2}$. In view of (1) Lemma \ref{lem-26}, weak solutions $\Gamma_{s,t-s}$ monotonically converge to a weak solution $\Gamma_t$ as $s\to-\infty$. 
\begin{equation}\label{eq:Gamma_conv}
  \Gamma_{s,t-s}\downarrow \Gamma_t,\qquad \text{as}\,\, s\to -\infty.
\end{equation}
Since $\lim_{s\to -\infty} \text{Vol}(\hat  \Gamma_{s,t-s})=-t\, \omega_n -{2V_{\Omega}}$, $\hat  \Gamma_t$ has non empty interior for $t<-\frac{2V_{\Omega}}{\omega_n}$ and has empty interior for $t>-\frac{2V_{\Omega}}{\omega_n}$.  This {\em defines  the ancient solution } $ \Gamma_t$ for $t\in(-\infty, - \frac{2V_{\Omega}}{\omega_n})$.
 
 \smallskip 
 
 We will now see that $\Gamma_t$ satisfies properties (i)-(iv) in the statement of our theorem. Properties (i)  and 
 the refection symmetry property (ii) clearly hold by  construction.  Furthermore,  property (iv) is just a consequence of Lemma \ref{lem-harnack}.
It remains to show property (iii).  By  construction, $\hat \Gamma_t$ is contained in $\hat \Gamma_{t,0}$ which is contained  in $\Omega\times \mathbb{R}$. Hence, $ \hat\Gamma_t$ is contained in $\Omega\times \mathbb{R}$. Suppose there is a smaller $\Omega' \subset \Omega$ such that $\hat  \Gamma_t$ is contained in $\Omega' \times \mathbb{R}$ for all $t$.  Since $\hat  \Gamma_t \subset  \hat  \Gamma_{t,0}$, $$ \sup_{\hat \Gamma_t} x_{n+1}  - \inf_{\hat \Gamma_t} x_{n+1} \le \sup_{\hat \Gamma_{t,0}} x_{n+1} - \inf_{\hat \Gamma_{t,0}} x_{n+1} = 2\lambda(-t).$$
 Therefore, $\text{Vol}(\hat  \Gamma_t) \le 2\lambda(-t) |\Omega'| = \frac{|\Omega'|}{|\Omega|} (-t\omega_n) $. On the other hand, we know that  $\text{Vol}(\hat  \Gamma_t) = -t\omega_n -{2V_{\Omega}}$. If $|\Omega'|<|\Omega|$, we have a contradiction by taking $t\to -\infty$ in the above inequality. This shows there is no such smaller $\Omega'$.

 \end{proof}
 \end{prop}

We next provide some extra properties of the solution $\Gamma_t$ constructed above. Those properties will be used in the proof of our uniqueness 
Theorem \ref{thm-compactunique}  in the next section. 

\begin{prop} \label{prop-23}
The constructed ancient solution $\Gamma_t$, $t \in(-\infty,T)$, with $T:=-\frac{2V_{\Omega}}{\omega_n}$,  satisfies $\big (\sup_{ \Gamma_t} |x_{n+1} |- \lambda |t| \, \big ) \uparrow 0 $   as $t\to -\infty$.

\begin{proof} The Harnack inequality implies that  the speed of each  tip of $\Gamma_t$  is  greater than $\lambda$,  hence the quantity   $\sup_{ \Gamma_t} |x_{n+1} |- \lambda |t| $ decreases, as $t$ increases. In addition, the solution  $\hat \Gamma_t$ is contained in $\hat \Gamma_{t,0}$, by construction, where $\hat \Gamma_{t,0}$ is given in the proof of Proposition \ref{prop-existence}. Hence $\sup_{ \Gamma_t} |x_{n+1} | \le  \lambda\, |t| $,  i.e. $\lim_{t\to-\infty} (\sup_{ \Gamma_t} |x_{n+1} |- \lambda |t| ) :=L \le 0$. If  $L<0$, this implies  that  the solution $\hat \Gamma_t$ is  contained in $\Omega \times [-(L+\lambda |t|), L+\lambda|t|]$. By the convergence of solution to the translating soliton around the tips, we  have  that \be \label{eq-1dr}\limsup_{t\to -\infty}\text{Vol}(\hat \Gamma_t)-   2[(L+\lambda |t|)|\Omega| - V_\Omega] \le 0.\ee Since   $2[(L+\lambda |t|)|\Omega| - V_\Omega]=  \omega_n|t| +2L|\Omega|-2V_\Omega$, this contradicts Proposition \ref{prop-existence} (i). This proves the assertion. \end{proof}
\end{prop}

\end{proof}

\section{Uniqueness of compact ancient solution} \label{sec-noncompact}

Given a bounded domain $\Omega \in \R^{n}$ we recall that $u_\Omega$ denotes the translator associated with the domain $\Omega$ satisfying $\inf_{x\in \Omega} u_\Omega(x)=0$.  Let us recall that if $\Omega$ is $C^{1,1}$ $$V_\Omega:= \int_\Omega u_\Omega(x) \, dx<\infty $$
by Lemma \ref{lem-finitevol}. We have shown in the previous section that there exists a compact ancient solution $\Gamma_t$, $t\in (- \infty, -\frac{2V_{\Omega}}{\omega_n})$ which is asymptotic to the cylinder $\Omega \times \R$ and which becomes extinct at time $T:= -\frac{2V_{\Omega}}{\omega_n}$.
We will next show that $\Gamma_t$ is unique up to translations in space along the axis $e^{n+1}$ and translations in time.

\smallskip

Let us briefly {\em outline the proof}  of the Theorem which will be given below.  As we stated in  Theorem \ref{thm-compactunique}, our
goal  is to show that  any  given compact ancient solution $\Sigma_t$ asymptotic to $\Omega\times \mathbb{R}$ which 
becomes extinct  at time $t=-\frac{2 V_\Omega}{\omega_n}$ is same as $\Gamma_t$, the solution constructed 
in the previous section, up to a translation in $e_{n+1}$ direction. For now, let us set aside to deal with this translation. The main step in our proof is to show the inclusion $$\hat \Gamma_t \subset \hat \Sigma_t, \qquad \mbox{for all} \,\, t <T.$$
Recall that   $\Gamma_t$ was obtained as the limit of $\Gamma_{s,t-s}$ as $s\to-\infty$,  where $\Gamma_{s,\tau}$ is the GCF running from $\Gamma_{s,0}$ and  $\Gamma_{s,0}$ is the compact surface obtained from the  gluing of two translators so that the distances from each tip to the origin is equal to $|s|\, \lambda$. Thus, it would have been   sufficient to  show that $\hat \Gamma_{s,0} \subset \hat \Sigma_s$,  for all $s\ll -1$. 
However, this is unlikely  to hold in general. Instead, it suffices to find a  family of convex sets  $\hat K_s \subset \hat \Sigma_s$ satisfying  $\hat K_s \subset \hat \Gamma_{s,0}$ and $\text{Vol}(\hat \Gamma_{s,0}\setminus \hat K_s) \to 0$,  as $s\to -\infty$. If $\hat K_{s,\tau}$ is the GCF from $K_{s}$, then $\hat K_{s,t-s} \subset \hat \Sigma_t$ for all $s\ll -1$.  Meanwhile, $\hat K_{s,t-s} \subset \hat \Gamma_{s,t-s}$ and $\text{Vol}(\hat K_{s,t-s})-\text{Vol}(\hat \Gamma_{s,t-s}) = \text{Vol}(\hat K_{s,0})-\text{Vol}(\hat \Gamma_{s,0}) \to 0$ as $s\to-\infty$, showing that $\hat K_{s,t-s} \to \hat \Gamma_{t}$ as $s\to -\infty$. This proves $\hat \Gamma_{t} \subset \hat \Sigma_t$. In this argument, we used the following two properties in a strong way:
 \begin{enumerate}[(i)]
\item  $\p_t (\text{Vol} ( \hat \Sigma_t))= -\omega_n$, for any GCF solution  $\Sigma_t= \p \hat \Sigma_t$, and
\item if two convex sets $M_1, M_2$ satisfy $\hat M_1 \subset \hat M_2$  and $\text{Vol}(\hat M_1) = \text{Vol}(\hat M_2)$, then $\hat M_1=\hat M_2$.
\end{enumerate} 

\smallskip 
Let us  next  describe  how we find such a family $\hat K_s$.  Instead of the translator $u_\Omega$  in the domain $\Omega$, 
we will consider  a hypersurface $x_{n+1}=u_\e(x)$ on $(1+\e)^{-\fr 1n} \Omega$ which is the translator of the same speed $\lambda$ on 
the domain $ (1+\e)^{-\fr 1n} \Omega \setminus B_\e(0)$ (see in Lemmas \ref{lem-32} and \ref{lem-33} below).  When the domain shrinks from
$\Omega$ to $ (1+\e)^{-\fr 1n} \Omega $, the associated translator speed larger that $\lambda$,  but we can adjust the speed to 
be equal to $\lambda$  by subtracting a small ball $B_\e(0)$ from  $(1+\e)^{-\fr 1n} \Omega$. lf we glue two  such hypersurfaces at  distance $|s|\, \lambda$, then the convergence of tip regions to the translator and the comparison principle from $-\infty$ time imply that $\Sigma_{s}$ contains such a hypersurface as 
$s \ll -1$ (see in Lemma \ref{lem-34}).   Let $\hat K^{\e_s}_{s}$ be the best possible  (meaning the smallest $\e_s$) convex set which can be inserted in 
$\hat \Sigma_s$ by the argument above. We want $\text{Vol}(\hat \Gamma_{s,0} \setminus \hat K^{\e_s}_s) \to 0$.  Roughly, $$\text{Vol}
\big (\hat \Gamma_{s,0} \big )\approx \text{Vol} \big (\Omega\times [-|s|\lambda,|s|\lambda] \big ) - 2V_\Omega= 2\lambda|s|\text{Vol}
\big (\Omega \big )- 2V_\Omega$$ and $$\ba\text {Vol} \big (\hat K^{\e_s}_{s} \big )&\approx \text{Vol} \big ((1+\e_s)^{-\fr1n}\Omega\times [-|s|\lambda,|s|\lambda] \big ) - 2V_{\e_s}\\&= 2\lambda|s| \text{Vol} \big ((1+\e_s)^{-\fr1n}\Omega \big )- 2V_{\e_s}.\ea$$
 Here, $V_\e$ denotes  the volume under the surface $x_{n+1}=u_\e(x)$ and it converges to $V_\Omega$,  as $\e\to 0$ (see in Lemma \ref{lem-32}). Since $\text{Vol}((1+\e_s)^{-\fr1n}\Omega )\approx (1-  \e_s) \text{Vol}( \Omega)$ for small $\e_s$, we need  $\e_s = o(|s|^{-1})$  to approximate the volume of $\hat \Gamma_{s,0}$ by $K^{\e_s}_s$ as $s\to-\infty$. A stronger statement of this assertion will be shown in Proposition \ref{prop-723}.  \smallskip
\smallskip

We will now give  the detailed proof of Theorem \ref{thm-compactunique}. Without loss of generality   we may assume that $0\in \Omega$ and
that $\inf  u_{\Omega}(\cdot )=u_{\Omega}(0)=0$. Let us fix   $r_0>0, R_0 >0$ such that $B_{2r_0}(0) \subset \Omega \subset B_{R_0}(0)$.  We begin with  a few preliminary results, where 
$\eta$ denotes a standard  cut off function  supported in $B_1\subset \R^n$ such that $\int \eta \, d x =1$.  

\begin{lemma}\label{lem-32} Let $\e_0:=\min( \tfrac{r_0}2,1)$ and $\lambda = \fr{\omega_n}{2|\Omega|}$. 
Given $\e \in(0,\e_0) $, there is a unique convex solution $u_\e:\Omega_\epsilon :=(1+\e)^{-\frac{1}{n}} \, \Omega  \to \mathbb{R}$ to the  equation 
\be\label{eqn-epsilon}
\sqrt{1+|Du_\e(x)|^2} \, K(u_\e,x) = \lambda \, \big ( 1+\e^{1+n}\eta(\e^{-1}x)\big ), 
\ee
satisfying the conditions
\bee  \inf_{\Omega_\epsilon  }  u_\e =0 \qquad \text{and} \qquad Du_\e((1+\e)^{-\frac{1}{n}}\Omega )= \mathbb{R}^n.\eee
Moreover, $V_{\e}:= \int_{(1+\e)^{-\frac{1}{n}}\Omega} u_\e(x)dx \to V_\Omega$,  as $\e\to 0$. 
\begin{proof}
The result of Urbas in \cite{Ur1,Ur2} guarantee the existence of a unique  solution of equation \eqref{eqn-epsilon} satisfying the required conditions.   
In addition,   standard regularity estimates  for equations of Monge-Amp\`ere type  imply that as $\epsilon \to 0$, $u_\epsilon(x)$ converges to the translator  $u_\Omega(x)$ 
having $\inf_\Omega u(x) =0$ and the convergence is  in the $C^{\infty}_{loc}$ sense. The convergence of $V_\e \to V_\Omega$  easily follows,  since the proof of  Lemma \ref{lem-finitevol} can be applied uniformly to the solutions $u_\e$ and gives \[\sup_{\e<\min( \tfrac{r_0}2,1)} \int_{\{x\in(1+\e)^{-\fr1n} \Omega\,:\,
 \dist (x,\p (1+\e)^{-\frac1n}\Omega) \le \delta \} }u_\e (x) dx = o(1) \text{ as }\delta \to0.\] 
 \end{proof}
\end{lemma}

Since $\Sigma_t$ converges to the translating soliton near tip regions, there is $\tau_0 \ll  -1$ and $M>0$ such that 
\be \label{eq-ineqd}|Du^+(x,t)|, \, |Du^-(x,t)| \le M, \qquad \text{ on }B_{\e_0}(0) \times (-\infty, \tau_0].\ee 
In particular, this implies $|Du_\Omega(x)| \le M$ on $x\in B_{\e_0}(0)$. We will use $\tau_0$ and $M$ in the remaining of this section. Also recall that we 
have assumed $u_{\Omega}(0)=\inf u_{\Omega}(\cdot)=0$ in this section.

\begin{lemma} \label{lem-33} For $\e \in(0, \e_0)$, $u_\epsilon$ defined in Lemma \ref{lem-32} satisfies
\[u_\e(x)+ M \e \ge u_\Omega(x) \quad \text{for all } x\in (1+\e)^{-\frac{1}{n} } \Omega.\]
\begin{proof}
Note $u_\e(x)+M \e \ge M \e  \ge u(x)$ on $ B_\e$ and it becomes infinity  at  $\p (1+\e)^{-\fr1n} \Omega$. They are both solutions to the translating soliton equation of  speed  $\lambda$ in the  domain $ (1+\e)^{-\fr1n} \Omega \setminus B_\e$.
Hence, the comparison principle implies the lemma. \end{proof} 
\end{lemma}

\smallskip 

For the next lemma let us define $d(t)= \min (|u^+(0,t)|, |u^-(0,t)|)$. $u^+(0,t)$ and $u^-(0,t)$ are very similar to $h^+(t)$ and $h^-(t)$, respectively (recall Definition \ref{defn-hpm}) in the following sense: since $u^-(x,t)-h^{-}(t)$ and $u^+(x,t)-h^+(t)$ converges to $u_\Omega(x)$ and $-u_\Omega(x)$ as $t\to-\infty$, respectively,  and $\inf u_\Omega = u(0)=0$, we have $|u^-(0,t)-h^-(t)|=o(1)$ and $|u^+(0,t)-h^+(t)|=o(1)$ as $t\to-\infty$. Moreover, $\p_t u^-(0,t) \ge \lambda$ and $\p_t u^+(0,t) \le -\lambda$.
\begin{lemma} \label{lem-34} Let $\e \in(0, \e_0)$ be a fixed given number.
For a solution $\Sigma_t$ satisfying the assumptions of Theorem \ref{thm-compactunique},  set  $d(t)= \min (|u^+(0,t)|, |u^-(0,t)|)$. 
Then,   the solution  $\Sigma_t$ encloses the convex body \[\hat K_{t,\e}:= \{ (x',x'_{n+1}) \in\mathbb{R}^{n+1} \, :\,  |x'_{n+1}| \le  -  u_{\e}(x')  - M\e+ d(t) \}\]
for all $t\le \min(t_\e,\tau_0)$,   where $\tau_0$ is a time satisfying \eqref{eq-ineqd} and   $$ t_\e :=\sup \{\,t\,:\,( (1+\e)^{-\fr1n} \Omega )\times \{0\}\subset  \hat \Sigma_t \cap \{x_{n+1}=0\} \,\}.$$
\begin{proof}
We apply again  the comparison principle.  Since 
$\p_t u^-(0,t)\ge \lambda$ and $\p_t u^+(0,t) \le -\lambda$ (this is due to the convergence to the soliton shown in Theorem \ref{thm-asymptotic}  and the Harnack inequality  in Lemma \ref{lem-harnack} (i)),  
 we have  that $K_{t,\e} :=\partial \hat K_{t,\e} $ is a supersolution 
of the Gauss curvature flow except the cross-section $\Sigma_t \cap \{ x_{n+1} =0\}$ and the two tip regions which are components of $(B_{\e}(0) \times \mathbb{R} )\cap K_{t,\e}$.  From the choice of $\tau_0$ and $d(t)$, we have $u_\e (x) + M \e -d(t) \ge u^-(x,t) $ and $-u_\e (x)-M \e +d(t) \le u^+(x,t)$ on $(x,t)\in B_\e \times (-\infty,\tau_0]$. Thus if $t\le \min(t_\e,\tau_0)$, then $K_{t,\e}$  does not touch to $\Sigma_t$ on these three regions. 

Moreover, the convergence to the translating soliton around tips and  Lemma \ref{lem-33} imply $\Sigma_t$ contains $K_{t,\e}$ for large negative times. By the comparison principle, we have $\hat K_{t,\e} \subset\hat \Sigma_t$ for $t< \min(t_\e, \tau_0)$. 
\end{proof}
\end{lemma}
A crucial ingredient  in the proof Theorem \ref{thm-compactunique} is  Proposition \ref{prop-723} which is an estimate on the rate that shows how fast $\Sigma _t$ becomes asymptotic to $\Omega\times \mathbb{R}$. We achieve this by a barrier argument. A one-parameter family of rotationally symmetric convex hypersurfaces in $\mathbb{R}^{n+1}$ represented by 
\[\Sigma_t = \{(x,x_{n+1})\in \mathbb{R}^{n+1}\, :\, |x| = r(x_{n+1},t) \}\] is a solution to the GCF if $r_{xx}<0$ and 
\begin{equation*}
r_{t} =  r_{xx}  (1+r_x^2)^{-\frac{n+1}{2}}  r^{-(n-1)}  .
\end{equation*}
Hence,   $\Sigma_t$ plays the role of an inner barrier if  
\begin{equation}\label{eq:inner.barrier}
r_{t} \leq  r_{xx}  (1+r_x^2)^{-\frac{n+1}{2}}  r^{-(n-1)}  .
\end{equation}

\begin{lemma}[Ancient entasis]\label{lem-innerbarrier}For given $\e>0$ and $L>1$, consider 1-parameter family of (incomplete) hypersurfaces $\{\Sigma^e_{t}\}_{t\le 0}$ defined by
\[\Sigma^e_t=\{ (x_1,\ldots,x_n,x_{n+1}) \,:\, |(x_1,\ldots,x_n)|=r(x_{n+1},t),\, |x_{n+1}| \le -t \},\] where $r(x,t)$ on $|x|\le -t$ is defined by  
\[ r(x,t)=  2\epsilon \left(2-  \exp {\frac{t}{L}} \cosh \frac {x}L\right ),\]
where $L = (2\e) ^{-(n-1)}$. Then $\{\Sigma^e_t\}_{t\le 0 }$  is an (incomplete) inner barrier to the GCF. Moreover, on its boundary $|x_{n+1}|=|t|$, there holds 
 \begin{equation}\label{eq:entasis_boundary}
2\e \le r(|t|,t)=r(-|t|,t)\le 3\e. 
 \end{equation}

\begin{remark}
Note that $\Sigma^e_t$ has rotationally symmetry about $x_{n+1}$-axis and has reflection symmetry about $\{x_{n+1}=0\}$.
\end{remark}
\begin{proof} Since $r_{xx}\le0$, $(1+r_x^2)\ge 1$, and $r \ge 2\e$, we have 
\[-  r_{xx}  (1+r_x^2)^{-\frac{n+1}{2}}  r^{-(n-1)}  \le - (2\e )^{-(n-1)} r_{xx} . \] Therefore, combining with $L= (2\e) ^{-(n-1)}$ yields \eqref{eq:inner.barrier}. In addition, we can obtain \eqref{eq:entasis_boundary}  by observing $\frac{1}{2} \leq e^{\frac{t}{L} } \cosh (\frac{t}{L}) \leq 1$.
\end{proof}

\end{lemma}

\bigskip

\begin{proposition}\label{prop-723}Let $\Omega$ be a bounded convex domain with $C^{1,1}$ boundary and let $\Sigma_t=\partial \hat \Sigma_t$, for $t\in(-\infty, T)$, be a compact ancient solution which is  asymptotic to $\Omega\times \mathbb{R}$. Then there exist large positive $C$, $L<\infty$ such that \[\sup_{x\in \partial \Omega\times \{0\}} d(x,\Sigma_t\cap \{x_{n+1}=0\} ) \le Ce^{t/L}.\] Here $L$ depends only on $\Omega$, and $C$ depends only on $\Omega$, $\Sigma_t$.  
\end{proposition}

\begin{proof} 
After a rescaling, we may assume $|\Omega|=\frac{\omega_n}{4}$ so that the translator asymptotic to $\Omega \times \mathbb{R}$ has the speed  $\lambda=2$. Then, by the Harnack Lemma \ref{lem-harnack} (i), we have
\begin{align}\label{eq:graph_speed}
&|\partial_t u^+(x,t) |\geq 2, && |\partial_t  u^- (x,t)|\geq 2,
\end{align}
where $u^\pm(x,t)$ are the maximum and minimum height functions of $\Sigma_t$ from Definition \ref{def-32}.

\bigskip

Next, due to the $C^{1,1}$ boundary assumption, there is some constant $\e>0$ such that at each point $p\in \partial \Omega$, we have $B_{6\e}(p+6\e n_p)\subset \Omega$, where $n_p$ denotes the unit inward pointing normal vector to $\partial\Omega$ at $p$. Namely, given $p \in \partial \Omega$,  $\Omega$ has an inscribed ball of radius $6\e$ tangent at $p$.

 We  denote the $0$-level set of $\hat \Sigma_t$ by  $\tilde \Omega_t:=\{x\in \mathbb{R}^n:(x,0)\in \hat \Sigma_t \}$. Since $\tilde \Omega_t$ increases to $\Omega$ as $t\to -\infty$, given $\delta\in (0,\e]$ there is some $t_\delta \ll T$ such that
\be\label{eq-omegat0}\sup_{x\in \partial\Omega} d(x,\tilde \Omega_{t_\delta }) \le \delta . \ee
 
\bigskip

Now, we claim that given $p\in \partial\Omega$ and $\delta\in (0,\e]$, $ \Sigma_t$ encloses $\Sigma^e_{t -t_{\e}}+ (p_\delta ,0)$ for $t\leq t_{\e}$, where $p_\delta:=p+(4\e +\delta)n_p$ and $\Sigma^e_t$ is the ancient entasis defined in Lemma \ref{lem-innerbarrier}. To prove the claim, we will apply the comparison principle by showing that the Dirichlet and initial boundaries of $\Sigma^e_{t -t_{\e}}+ (p_\delta ,0)$ are enclosed by $\Sigma_t$.

\bigskip

We first consider the Dirichlet condition. By \eqref{eq:entasis_boundary}, we have $$\partial \Sigma^e_{t -t_{\e}}+ (p_\delta ,0) \subset \overline{B_{3\e}(p_\delta)} \times \{\pm |t-t_{\e}|\}.$$ Therefore, it is enough to show 
\begin{equation}\label{eq:Lateral}
\overline{B_{3\e}(p_\delta)} \times [-|t-t_{\e}|,|t-t_{\e}|] \subset \hat \Sigma_t,
\end{equation}
for $t\leq t_{\e}$. Indeed, \eqref{eq-omegat0} implies $\overline{B_{3\e}(p_\delta)} \subset \tilde \Omega_{t_{\e}}$, and thus we have $u^+(x,t_{\e})\geq 0$ and $u^-(x,t_{\e})\leq 0$ for $|x-p_\delta|\leq 3\e$. Therefore, \eqref{eq:graph_speed} yields $u^+(x,t)\geq 2|t-t_{\e}|$ and $u^-(x,t_{\e})\leq -2|t-t_{\e}|$ in $\overline{B_{3\e}(p_\delta)}$ for $t\leq t_{\e}$. This gives us \eqref{eq:Lateral}.

\bigskip

To deal with the initial condition, we notice that 
\begin{equation*}
\Sigma^e_{t-t_{\e}}+ (p_\delta ,0)\subset B_{4\e}(p_\delta) \times [-|t-t_{\e}|,|t-t_{\e}|].
\end{equation*}
On the other hand, \eqref{eq-omegat0} yields $B_{4\e}(p_\delta) \subset \tilde \Omega_{t_{\delta}}$ so that we obtain $u^+(x,t_{\delta})\geq 0$ and $u^-(x,t_{\delta })\leq 0$ in $B_{4\e}(p_\delta)$. Hence,  \eqref{eq:graph_speed} again leads to $u^+(x,t)\geq 2|t-t_{\delta }|$ and $u^-(x,t_{\e})\leq -2|t-t_{\delta}|$ in $\overline{B_{3\e}(p_\delta)} \times (-\infty, t_{\delta}] $. Therefore, there is some $\bar  t \ll t_\e $ depending only on $t_\e,t_\delta$ such that
\begin{equation*}
B_{4\e}(p_\delta) \times [-|t-t_{\e}|,|t-t_{\e}|] \subset \hat \Sigma_t
\end{equation*}
holds for $t\leq \bar  t$.

\bigskip

Since $\Sigma^e_{t-t_{\e}}+ (p_\delta ,0)$ satisfies the boundary conditions as an inner barrier of $\Sigma_t$, the comparison principle guarantees
$\Sigma^e_{t-t_{\e}}+ (p_\delta ,0) \subset \hat \Sigma_t$ for $t\leq t_\e$. Passing $\delta$ to $0$ and interpreting the containment among zero level sets, we conclude 
\begin{equation*}
d(p, \tilde \Omega_t)\leq  d(p, \pi (\Sigma^e_{t-t_{\e}}) + p+4\e n_p )=4\e - r(0,t-t_\e)=  2\e e^{\frac{t-t_\e}{L}} 
\end{equation*}
where $\pi (\Sigma^e_{t-t_{\e}})=\{x : (x,0)\in \Sigma^e_{t-t_{\e}}\}$. This completes the proof, since  $\e$, $L=(2\e)^{-(n-1)}$, and $t_\e$ are independent on $p$.
\end{proof}

\bigskip

\begin{remark} The exponential decay in Proposition \ref{prop-723} is sharp in the sense that the non-compact translators satisfies similar bounds (as in the proof of Lemma \ref{lem-finitevol}) and the Grim Reaper in $\mathbb{R}^2$\[\sin y = e^t \cosh x\] has no better decay. \end{remark}

Remembering $0\in  \Omega$, Proposition \ref{prop-723} implies that there exists large positive $L$ such that  \be\label{eqn-dist}
\e_t    := \inf \{\e>0:\,  (1+\e)^{-\frac1n} \Omega \subset \Omega_t \} =O(e^{t/L}), \qquad \mbox{as} \,\, t\to-\infty.
\ee 
Instead of this exponential decay,	$\e_t = o(|t|^{-1})$ will be sufficient to conclude the proof of Theorem \ref{thm-compactunique}.

\begin{proof}[Proof of Theorem \ref{thm-compactunique}] We begin by choosing  a number $t_1 \ll -1$ such that the cross-section  $\tilde \Omega_{t_1}:=\{x\in \mathbb{R}^n:(x,0)\in \hat \Sigma_{t_1} \}$ contains the origin, and hence $\e_t$ defined in \eqref{eqn-dist} is a finite number for each $t \leq t_1$.  Recall our  notation $\Gamma_t = \p \hat \Gamma_t$ and $\Sigma_t = \p \hat \Sigma_t$,  where $\hat \Gamma_t$ is given in Proposition \ref{prop-existence}. As the key step, we claim
\be\label{eqn-mstep1}
\hat \Gamma_{t-t_1}\subset \hat \Sigma_{t}, \qquad  \text{ for all } t< t_1- 2\omega_n^{-1}V_\Omega .
\ee

Note first that, by Lemma \ref{lem-34} and the definition of $\e_t$, we have that 
\begin{equation}\label{eq:containment1}
\hat K_{t,\e_t} \subset \hat\Sigma_t
\end{equation}
for $t\le \tau_1$ where $\tau_1:= \min(t_1,\, \tau_0, \,\sup\{t\,:\, \e_t < \e_0 )\})$. Here, $\tau_0$ and $\e _0$ are defined in \eqref{eq-ineqd} and Lemma \ref{lem-32}, respectively. Furthermore since  $\p_t u^-(0,t) \ge \lambda $ and $\p_t u^+(0,t) \le -\lambda$ by the Harnack, we obtain
\begin{equation}\label{eq:containment2}
\hat K'_{t} :=  \{ (x',x'_{n+1}):\,   | x'_{n+1} |\le  -u_{\e_t}(x') -M \e_t +\lambda|t-t_1| \,\} \subset \hat K_{t,\e_t},
\end{equation} 
for all $t\le \tau_1$. Next, recalling  $\Gamma_{s,0}$ from the proof of Proposition \ref{prop-existence}, we have 
\begin{equation}\label{eq:containment3}
\hat K'_{t} \subset \hat \Gamma_{t-t_1,0}
\end{equation}

for $t\le \tau_1$ by Lemma \ref{lem-33}. Moreover, we have\begin{equation}\label{eq:volume.convergence}
\text{Vol}(\hat \Gamma_{t-t_1,0} )-\text{Vol}(\hat K'_t) \to 0, \qquad \mbox{ as}\,\,t  \to -\infty.
\end{equation}Indeed, points in $\hat \Gamma_{t-t_1,0}\setminus\hat K'_t$  belong to one of two cylinders either $\Omega\setminus \Omega_{\e _t} \times [-\lambda |t-t_1|,\lambda |t-t_1|]$ or $\Omega_{\e _t}\times [-\lambda |t-t_1|,\lambda |t-t_1|]$. As a consequence of \eqref{eqn-dist} (which is a consequence of Proposition \ref{prop-723}), the volume of the former cylinder converges to zero. Moreover, the volume of $\hat \Gamma_{t-t_1,0}\setminus\hat K'_t$ inside of the later cylinder converges to zero since $\lim_{t\to -\infty} V_{\e_t}=V_\Omega$ by Lemma \ref{lem-32}. Now \eqref{eq:volume.convergence} follows by combining these two.

\bigskip

On the other hand, we denote the GCFs running from $K'_{t,0}=\p \hat K'_{t}$ and $\Gamma_{s,0}=\partial \hat \Gamma_{s,0}$  by $K'_{t,s}=\p \hat K'_{t,s}$ and $\Gamma_{s,t}=\partial \hat \Gamma_{s,t}$, respectively. Since \eqref{eq:containment1} and \eqref{eq:containment2} gave us $\hat K'_{t} \subset \hat K_{t,\e_t} \subset \hat \Sigma_t$,  the comparison principle implies \[\hat K'_{t,t_2-t}\subset \hat \Sigma_{t_2}, \qquad \text{for all }t_2\ge t.\]
We remember \eqref{eq:containment3} and apply \eqref{eq:volume.convergence} to obtain
\begin{equation*}
\text{Vol}(\hat \Gamma_{t-t_1,t_2-t} )-\text{Vol}(\hat K'_{t,t_2-t})= \text{Vol}(\hat \Gamma_{t-t_1,0} )-\text{Vol}(\hat K'_t)\to 0
\end{equation*}
as $t\to -\infty$. In addition, we have $\Gamma_{t-t_1,t_2-t} \downarrow \Gamma_{t_2-t_1}$ as $t\to-\infty$ by  \eqref{eq:Gamma_conv}. 
Therefore, we conclude that 

\begin{equation}\label{eq:conclusion}
\hat \Gamma_{t_2-t_1}\subset \hat \Sigma_{t_2}
\end{equation}
for all  $t_2< t_1-\frac{2V_\Omega}{\omega_n}$,  which proves our claim \eqref{eqn-mstep1}.

\bigskip
 We will now use the inclusion \eqref{eqn-mstep1} to conclude the proof of our uniqueness theorem. First, since $\hat \Gamma_{t-t_1} \subset \hat \Sigma_t $ and  $\p_t \text{Vol}(\hat \Sigma_t) = \p_t \text{Vol}(\hat \Gamma_{t-t_1})= - \omega_n$, we obtain  \be\label{eq-constantintime}\text{Vol}(\hat \Sigma_t \setminus \hat \Gamma_{t-t_1})\text{ is constant in time.}\ee We recall the definitions of heights $h^\pm$ of $\Sigma_t$ given in  Definition \ref{defn-hpm}. Since $\hat \Gamma_{t-t_1} \subset \hat \Sigma_t$, $\p_t h^+(t) \le -\lambda $  and $\p_t h^-(t)\ge  \lambda$, we obtain that, as $t\to -\infty$, $h^+(t)-\lambda |t|$ increases and $h^-(t)+\lambda|t|$ decreases.  In addition,  by Theorem \ref{thm-asymptotic}, Proposition \ref{prop-23}, and \eqref{eq-constantintime} we have 
\begin{align}\label{eq-h+h-}
&\lim_{t\to -\infty} h^+(t)-\lambda|t| =C^+, &&  \lim_{t\to -\infty}  h^-(t)+\lambda|t|=C^-,
\end{align}
for some constants $C^+$ and $C^-$.  Next, recall $\hat \Gamma_{t-t_1} \subset \hat \Gamma_{t-t_1,0}$ by \eqref{eq:Gamma_conv} and $\mathrm{Vol}(\hat \Gamma_{t-t_1,0}\setminus \hat \Gamma _{t-t_1})$ converges to $0$ as $t\to -\infty$ by \eqref{eq:volume0} and Proposition \ref{prop-existence} (i). Thus $\text{Vol}(\hat \Sigma_t \setminus \hat \Gamma_{t-t_1})=\text{Vol}(\hat \Sigma_t \setminus \hat \Gamma_{t-t_1,0})+o(1)$ as $t\to-\infty$. Observe  \begin{equation*}\ba 
\text{Vol}(\hat \Sigma_t \setminus \hat \Gamma_{t-t_1,0})= \int  [(u_{\Omega}(x)-\lambda |t-t_1|)-u^-(x,t)]_+ dx\\ +\int [u^+(x,t)+(u_{\Omega}(x)-\lambda |t-t_1|)]_+dx \ea \end{equation*}
where we interpret $u^-=\infty$ and $u^+=-\infty$ outside of their domain.  By the dominated convergence theorem, the two integrals above converge to $(-C^--\lambda t_1)|\Omega|$ and $(C^+-\lambda t_1)|\Omega|$, respectively, as $t \to -\infty$. Indeed, by \eqref{eq-h+h-}, the first integrand is bounded by $u_{\Omega}(x) -C^- -\lambda t_1 \in L^1(\Omega)$ by Lemma \ref{lem-finitevol} and the integrand converges locally uniformly to the constant $-C^- -\lambda t_1$ by Theorem \ref{thm-asymptotic}. The same argument works for the second integral. This proves \begin{equation*}\text{Vol}(\hat \Sigma_t \setminus \hat \Gamma_{t-t_1,0})=
(C^+-C^--2\lambda t_1)|\Omega|+o(1),
\end{equation*} and we conclude $\text{Vol}(\hat \Sigma_t \setminus \hat \Gamma_{t-t_1})=(C^+-C^--2\lambda t_1)|\Omega|$ as this is constant. Since both of $\Sigma_t$ and $\Gamma_t$ become extinct at $T=-\frac{2V_\Omega}{\omega_n}$, we have
\begin{equation*}
(C^+-C^--2\lambda t_1)|\Omega|=\text{Vol}(\hat \Sigma_{T+t_1})=-\omega_nt_1,
\end{equation*}
namely $C^+=C^-=:C$.

 \smallskip 
It remains to  show that $\Sigma_t=\Gamma_t + C \, e_{n+1}$ and the proof follows from  what we have already done.We may summarize the first part of the proof as follows:  if there exist constants $t_0$, $\tau_1$,  and a decreasing function $\e_t =(|t|^{-1})$  as $t\to -\infty$, such that  $\hat\Sigma_t$ contains
 \[\hat K_{t,\e_t,t_0}:= \{ (x,x_{n+1}):\, | x_{n+1}| \le -u_{\e_t}(x) -M\e_t +\lambda |t-t_0| \,\}\]                  
 for all $t<\tau_1$, then $\hat \Gamma_{t-t_0} \subset \hat \Sigma_t$. 
By looking at the intersection between $\Sigma_t$ and  $\{x_{n+1}=C\}$, we may define $\e_t'$ in the same way as $\e_t$ is defined in \eqref{eqn-dist},
that is  \[\e'_t    := \inf \{\e>0:\,  (1+\e)^{-\frac1n} \Omega \subset \Omega'_t \} \]
where $\Omega'_t$ is the cross-section of $\hat \Sigma_t$ by $\{ x_{n+1} = C \}$. Similarly to  Proposition \ref{prop-723}, we have $\e'_t = o(|t|^{-1})$. Next, for each small $\delta>0$,  Lemma \ref{lem-34} and  \eqref{eq-h+h-} imply that  there is $\tau_\delta \ll -1$ such that $\hat K_{t,\e'_t, -\delta}+C\e_{n+1} \subset \hat \Sigma_t$,  for all $t< \tau_\delta$.
Then, as we obtained \eqref{eq:conclusion}, we can derive $\hat \Gamma_{t+\delta}+C e_{n+1} \subset \hat \Sigma_t$ by repeating the argument from \eqref{eq:containment2} to \eqref{eq:conclusion} after replacing $t_1$ by $-\delta$.  Taking $\delta\to 0$, we get $\hat \Gamma_t+ Ce_{n+1}\subset \hat \Sigma_t$. Since $\text{Vol}(\hat \Gamma_t)=\text{Vol}(\hat \Sigma_t)= -\omega_n t -2V_\Omega$, we finally conclude that $\Gamma_t +C\, e_{n+1} =  \Sigma_t$. 

\end{proof}


\centerline{\bf Acknowledgements}

\smallskip 

\noindent B. Choi has been partially supported by NSF grant DMS-1600658 and the National Research Foundation of Korea grant NRF-2022R1C1C1013511. B. Choi also thanks Columbia University and University of Toronto where the research was initiated and then developed.

\noindent K. Choi has been partially supported by KIAS Individual Grant MG078901 and TJ Park Science Fellowship.

\noindent P. Daskalopoulos has been partially supported by NSF grant DMS-1600658 and DMS-1900702.

\bigskip
\bigskip


\begin{thebibliography}{99}


\bibitem{An94} B. Andrews.
\newblock Entropy inequalities for evolving hypersurfaces.
\newblock {\em Communications in Analysis and Geometry}, 2(1): 53--64, 1994.


\bibitem{An96} B. Andrews.
\newblock Contraction of convex hypersurfaces by their affine normal.
\newblock {\em Journal of Differential Geometry}, 43(2): 207--230, 1996.



\bibitem{An99} B. Andrews.
\newblock Gauss curvature flow: the fate of the rolling stones.
\newblock {\em Inventiones mathematicae}, 138(1): 151--161, 1999.

\bibitem{An00} B. Andrews.
\newblock Motion of hypersurfaces by Gauss curvature.
\newblock {\em Pacific Journal of Mathematics}, 195(1): 1--34, 2000.


\bibitem{An03} B. Andrews.
\newblock Classification of limiting shapes for isotropic curve flows.
\newblock {\em Journal of the American Mathematical Society}, 16(2): 443--459, 2003.

\bibitem{AC} B. Andrews and X. Chen.
\newblock Surfaces moving by powers of Gauss curvature.
\newblock {\em Pure and Applied Mathematics Quarterly}, 8(4): 825--834, 2012.


\bibitem{AGN} B. Andrews, P. Guan, and L. Ni.
\newblock Flow by powers of the Gauss curvature.
\newblock {\em Advances in Mathematics}, 299: 174--201, 2016.


\bibitem{ADS} S.B. Angenent, P. Daskalopoulos, and N. Sesum. 
\newblock Uniqueness of two-convex closed ancient solutions to the mean curvature flow.
\newblock {\em Ann. Math.}, 192.2 (2020): 353-436. 

\bibitem{BLT0} T. Bourni, M. Langford, and G. Tinaglia. 
\newblock A collapsing ancient solution of mean curvature flow in $\mathbb {R}^ 3$.
\newblock {\em To appear in Journal of Differential Geometry} (arXiv:1705.06981 (2017)).


\bibitem{BLT} T. Bourni, M. Langford, and G. Tinaglia. 
\newblock Convex ancient solutions to curve shortening flow.
\newblock {\em Calc. Var. PDE.}, 59.4 (2020): 1-15.



\bibitem{BCD} S. Brendle, K. Choi and P. Daskalopoulos.
\newblock Asymptotic behavior of flows by powers of the Gaussian curvature.
\newblock {\em Acta Mathematica}, 219(1): 1--16, 2017.



\bibitem{B} S. Brendle.
\newblock Ancient solutions to the Ricci flow in dimension 3.
\newblock {\em Acta Math.}, 225.1 (2020): 1-102.

\bibitem{BC} S. Brendle and K. Choi.
\newblock Uniqueness of convex ancient solutions to mean curvature flow in ${\mathbb {R}}^ 3$.
\newblock {\em Invent. Math.}, 217.1 (2019): 35-76.


\bibitem{Ca58} E. Calabi.
\newblock Improper affine hyperspheres of convex type and a generalization of a theorem by K.J\"orgens. 
\newblock {\em Michigan Math. J.} 5: 105--126, 1958.


\bibitem{Ca72} E. Calabi.
\newblock Complete affine hyperspheres. I.
\newblock {\em Symposia Mathematica} Vol. X 19--38 Academic Press, London,
1972.




\bibitem{CY} S.Y. Cheng and S.T. Yau.
\newblock Complete affine hypersurfaces. I. The completeness of affine metrics.
\newblock {\em Communications on Pure and Applied Mathematics}, 39(6): 839–-866, 1986.




\bibitem{CCD} B. Choi, K. Choi and P. Daskalopoulos.
\newblock Convergence of curve shortening flow to translating solitons
\newblock {\em   Amer. J. Math.}, 143.4 (2021): 1043-1077.

\bibitem{CCD2} B. Choi, K. Choi and P. Daskalopoulos.
\newblock  Convergence  of   Gauss curvature  flows  to translating solitons.
\newblock {\em Adv. Math.},  397 (2022): 108207. 



\bibitem{CD} K. Choi and P. Daskalopoulos.
\newblock Uniqueness of closed self-similar solutions to the Gauss curvature flow.
\newblock arXiv:1609.05487, 2016.


\bibitem{CDL} K. Choi, P. Daskalopoulos, and K.A. Lee. 
\newblock Translating solutions to the Gauss curvature flow with flat sides.
\newblock {\em Analysis \& PDE} 14.2 (2021): 595-616. 


\bibitem{CDKL} K. Choi, P. Daskalopoulos, L. Kim, and K.A. Lee. 
\newblock The evolution of complete non-compact graphs by powers of Gauss curvature.
\newblock  {\em Journal f{\"u}r die reine und angewandte Mathematik}, DOI: 10.1515/crelle-
2017--0032, 2017.

\bibitem{CHH} K. Choi, R. Haslhofer and O. Hershkovits.
\newblock Ancient low entropy flows, mean convex neighborhoods, and uniqueness.
\newblock {\em To appear in Acta Math.}, arXiv:1810.08467 (2018).



\bibitem{Ch85} B. Chow.
\newblock Deforming convex hypersurfaces by the $n$th root of the Gaussian curvature.
\newblock {\em Journal of Differential Geometry} 22(1): 117--138, 1985.



\bibitem{Ch91} B. Chow. 
\newblock On Harnack's inequality and entropy for the Gaussian curvature flow.
\newblock {\em Communications on Pure and Applied Mathematics}, 44(4): 469--483, 1991.


\bibitem{DHS} P. Daskalopoulos, R. Hamilton and N. Sesum. 
\newblock Classification of compact ancient solutions to the curve shortening flow.
\newblock {\em Journal of Differential Geometry}, 84.3 (2010): 455-464.



\bibitem{DHS2} P. Daskalopoulos, R. Hamilton and N. Sesum. 
\newblock Classification of ancient compact solutions to the Ricci flow on surfaces.
\newblock {\em Journal of Differential Geometry}, 91.2 (2012): 171-214.



\bibitem{DH} P. Daskalopoulos and R. Hamilton. 
\newblock The free boundary in the Gauss curvature flow with flat sides.
\newblock  {\em Journal für die reine und angewandte Mathematik}, 510: 187--227, 1999.




\bibitem{DL} P. Daskalopoulos and K.A. Lee.
\newblock Worn stones with flat sides all time regularity of the interface.
\newblock {\em Invent. Math.},  156(3): 445--493, 2004.


\bibitem{DS} P. Daskalopoulos and O. Savin. 
\newblock $C^{1,\alpha}$  regularity of solutions to parabolic Monge-Amp\`ere equations.
\newblock {\em Amer. J. Math},  134(4): 1051--1087, 2012.

\bibitem{DS2} P. Daskalopoulos and N. Sesum. 
\newblock Uniqueness of ancient compact non-collapsed solutions to the 3-dimensional Ricci flow.
\newblock arXiv preprint arXiv:1907.01928 (2019).


\bibitem{Fi} W. Firey. 
\newblock Shapes of worn stones.
\newblock {\em Mathematika}, 21(1): 1--11, 1974.

\bibitem{GN} P. Guan and L. Ni. 
\newblock Entropy and a convergence theorem for Gauss curvature flow in
high dimension.
\newblock {\em Journal of European Mathematical Society}, 19(12):  3735--3761, 2017.

\bibitem{Ha94} R. Hamilton. 
\newblock Worn stones with flat sides.
\newblock {\em Discourses Math. Appl}, 3: 69--78, 1994.

\bibitem{Ha95} R. Hamilton. 
\newblock Harnack estimate for the mean curvature flow.
\newblock {\em Journal of Differential Geometry}, 41(1): 215--226, 1995.


\bibitem{JW} H. Jian and X.J. Wang.
\newblock Existence of entire solutions to the Monge-Amp\`ere equation.
\newblock {\em Amer. J. Math.},  136(4): 1093--1106, 2014.

\bibitem{Jo} K. J\"orgens. 
\newblock \"{U}ber die L\"osungen der Differentialgleichung $rt −s^2=1$.
\newblock {\em Mathematische Annalen}, 127: 130--134, 1954.


\bibitem{Ni} J.C.C. Nitsche. 
\newblock Elementary proof of Bernstein’s theorem on minimal surfaces.
\newblock {\em Annals of Mathematics}, 66(2): 543--544, 1957.


\bibitem{Po} A.V. Pogorelov. 
\newblock On the improper convex affine hyperspheres.
\newblock {\em Geometriae Dedicata}, 1(1): 33–-46, 1972.

\bibitem{Ts} K. Tso. 
\newblock Deforming a hypersurface by its Gauss-Kronecker curvature.
\newblock {\em Communications on Pure and Applied Mathematics}, 38(6): 867--882, 1985.



\bibitem{Ur1} J. Urbas. 
\newblock Global H\"older estimates for equations of Monge-Amp\`ere type.
\newblock {\em Invent. Math.}, 91(1): 1--29, 1988.

\bibitem{Ur2} J. Urbas. 
\newblock Complete noncompact self-similar solutions of Gauss curvature flows I. Positive powers.
\newblock {\em Mathematische Annalen}, 311(2): 251--274, 1998.

\end{thebibliography}
\end{document}